\documentclass[a4paper,10pt]{amsart}

\usepackage{acatta}

\newcommand{\blankSpace}{\cdot}

\newcommand{\kod}{\operatorname{kod}}

\newcommand{\const}{\text{const}}
\newcommand\C{{\mathbb C}}
\newcommand\Q{{\mathbb Q}}
\newcommand\R{{\mathbb R}}

\newcommand{\elm}[3]{\left( \begin{array}{cccc} e^{k #3} & 0 & 0 & #1\\ 0 & e^{-k #3} & 0 & #2\\ 0 & 0 & 1 & #3\\ 0 & 0 & 0 & 1 \end{array} \right)}

\renewcommand{\Im}{\operatorname{Im}}
\renewcommand{\Re}{\operatorname{Re}}

\title[Kodaira dimension of almost complex 4-dimensional solvmanifolds]{On Kodaira dimension of almost complex 4-dimensional solvmanifolds without complex structures}

\author{Andrea Cattaneo}
\address{Dipartimento di Matematica ``Tullio Levi-Civita''\\
Universit\`a degli Studi di Padova\\
Via Trieste 63\\
35121 Padova, Italy.\newline Current affiliation: Dipartimento di Scienze Matematiche, Fisiche e Informatiche
Unit\`a di Matematica e Informatica\\
Universit\`a degli Studi di Parma\\
Parco Area delle Scienze 53/A, 43124\\
Parma, Italy.}
\email{andrea.cattaneo@unipd.it\\andrea.cattaneo@unipr.it}

\author{Antonella Nannicini}
\address{Dipartimento di Matematica ed Informatica ``U. Dini''\\
Universit\`a degli Studi di Firenze\\
Viale Morgagni 67/A\\
50134 Firenze, Italy.}
\email{antonella.nannicini@unifi.it}

\author{Adriano Tomassini}
\address{Dipartimento di Scienze Matematiche, Fisiche e Informatiche
Unit\`a di Matematica e Informatica\\
Universit\`a degli Studi di Parma\\
Parco Area delle Scienze 53/A, 43124\\
Parma, Italy.}
\email{adriano.tomassini@unipr.it}

\keywords{Kodaira dimension, almost complex manifolds, almost K\"ahler manifolds, canonical connection.}
\thanks{This work was partially supported by the Project PRIN 2017 ``Real and Complex Manifolds: Topology, Geometry and holomorphic dynamics'' and by GNSAGA of INdAM. When the major part of this work was done Andrea Cattaneo was supported by the `Grant de Bartolomeis', a fellowship of Universit\`a di Firenze in memory of Prof. Paolo de Bartolomeis}
\subjclass[2010]{53C55, 53C25}

\begin{document}

\begin{abstract}
The aim of this paper is to continue the study of Kodaira dimension for almost complex manifolds, focusing on the case of compact $4$-dimensional solvmanifolds without any integrable almost complex structure. According to the classification theory we consider: $\mathfrak{r}\mathfrak{r}_{3, -1}$, $\mathfrak{nil}^4$ and $\mathfrak{r}_{4, \lambda, -(1 + \lambda)}$ with $-1 < \lambda < -\frac{1}{2}$. For the first solvmanifold we introduce special families of almost complex structures, compute the corresponding Kodaira dimension and show that it is no longer a deformation invariant. Moreover we prove Ricci flatness of the canonical connection for the almost K\"ahler structure. Regarding the other two manifolds we compute the Kodaira dimension for certain almost complex structures. Finally we construct a natural hypercomplex structure providing a twistorial description.
\end{abstract}

\maketitle

\tableofcontents

\section{Introduction}
Let $M$ be a smooth $2n$-dimensional manifold endowed with an almost complex structure $J$, that is a 
smooth $(1, 1)$-tensor field such that $J^2 = -\id_M$. According to the celebrated Newlander--Nirenberg Theorem, the almost complex structure $J$ is induced by a structure of complex manifold on $M$, that is $J$ is {\em integrable}, if and only if the torsion tensor of $J$, namely the {\em Nijenhuis tensor} of $J$, vanishes. The investigation of the existence of new almost complex invariants of $(M, J)$ is a natural problem in almost complex geometry. Recently, in this direction, H.~Chen and W.~Zhang in \cite{Chen-Zhang-1} and \cite{Chen-Zhang-2}, starting with a $2n$-dimensional compact almost complex manifold $(M, J)$, introduce the notion of {\em plurigenera} and give the definition of {\em Kodaira dimension} of $(M, J)$, denoted as $P_m(M, J)$, $\kod(M, J)$ respectively. They prove that, given any $4$-dimensional compact almost complex manifold, the plurigenera and the Kodaira dimension are birational invariants in the almost complex category. Furthermore, they show that such invariants are not stable under small deformations of the almost complex structure. For this purpose, they construct in \cite[$\S$6.1]{Chen-Zhang-1} a $1$-parameter family $\{ J_a \}_{a \in \R^*}$ of almost complex structures on the Kodaira--Thurston manifold $X$ such that
\[\kod(X, J_a) =
\left\{
\begin{array}{ll}
0 & a \in \pi \Q \smallsetminus \{ 0 \},\\
-\infty & a \notin \pi \Q.
\end{array}
\right.\]
In \cite{Cattaneo-Nannicini-Tomassini} the authors compute the Kodaira dimension of special families of $6$-dimensional almost complex manifolds and of their deformations. More in particular, starting with the smooth manifold underlying the Nakamura complex $3$-fold, they compute the Kodaira dimension of an almost K\"ahler deformation, showing that the possible values are $0$ or $-\infty$; furthermore, it is proved that the Ricci curvature of the canonical connection $\nabla^c$ vanishes. 

The aim of this paper is to continue the study of the Kodaira dimension in the almost complex category. More specifically, we focus on compact $4$-dimensional solvmanifolds $M$, that is, compact quotients of simply connected solvable Lie groups by discrete cocompact subgroups, i.e., by lattices, without any complex structure. According to the classification theory (see \cite{Angella-Bazzoni-Parton}) these are the solvable Lie algebras whose corresponding simply connected Lie groups admit compact quotients by a lattice:
\begin{enumerate}
\item $\mathfrak{r}\mathfrak{r}_{3, -1}$;
\item $\mathfrak{nil}^4$ ;
\item $\mathfrak{r}_{4, \lambda, -(1 + \lambda)}$ with $-1 < \lambda \leq -\frac{1}{2}$.
\end{enumerate}
Notice that, with the only exception of $\mathfrak{r}_{4, -\frac{1}{2}, -\frac{1}{2}}$, none of the $4$-dimensional compact quotients as above carries any integrable almost complex structure. Therefore, the notion by Chen and Zhang of Kodaira dimension appears as one useful tool for the study of the almost complex geometry of such manifolds.

The paper is organized as follows. After we recall the definition of Kodaira dimension for almost complex manifolds in Section \ref{preliminaries}, in Section \ref{almost-complex-solvmanifolds} we consider the families of solvmanifolds $M(k)$, $\cN$ and $\cM(\lambda)$ associated respectively to the solvable Lie algebras $\mathfrak{r}\mathfrak{r}_{3, -1}$, $\mathfrak{nil}^4$ and $\mathfrak{r}_{4, \lambda, -(1 + \lambda)}$ with $-1 < \lambda < -\frac{1}{2}$ and we recall their constructions. Sections \ref{s5}, \ref{s6} and \ref{s7} are devoted to the computations of the Kodaira dimension of special families of almost complex structures on $M(k)$, $\cN$ and $\cM(\lambda)$: see Propositions \ref{prop: kod dim m(k) 1} and \ref{prop: kod dim m(k) 2} for the structures on $M(k)$, Propositions \ref{prop: kod dim n 1}, \ref{prop: kod dim n 2} and \ref{prop: kod dim n 3} for the structures on $\cN$ and Proposition \ref{prop: kod dim m(lambda)} for the one on $\cM(\lambda)$. The main tool for our computations is an ad hoc method involving Fourier series: for each example we view functions on the solvmanifolds as functions defined on the universal cover which are invariant with respect to a suitable lattice. Then we find a sublattice with respect to which such functions are periodic, and then we express them in Fourier series. Finally, we compute the Fourier coefficients by solving some system of partial differential equations and then obtain the Kodaira dimension.

Furthermore, the Ricci curvature of the canonical Chern connection is computed in the case of almost K\"ahler structures. In particular, for the almost K\"ahler family $(J_s, g_s)$ on $M(k)$ (see \eqref{eq: first j}), the Ricci curvature vanishes (see Proposition \ref{Chern-flat}). It is worth to note that there exists a non Chern--Ricci flat almost K\"ahler metric, with vanishing scalar curvature on $(\cN, J, g_J)$, such that $\kod(\cN, J) = -\infty$ (see Proposition \ref{scalar-flat}). Notice that the Kodaira dimension of $M(k)$ is no longer a deformation invariant.

In Section \ref{sect: twistor}, we construct a natural hypercomplex structure providing a twistorial description. Finally, in Section \ref{sect: norden} we describe some natural Norden structures.

\begin{ack}
The authors are grateful to Weiyi Zhang for a interesting discussions and suggestions during the workshop `Geometria in Bicocca 2019' held in Milan, which led to the preparation of this paper. They also want to thank Scott O.\ Wilson for having attracted their attention on an inaccuracy in $\S$ \ref{sect: M(lambda)}. The authors want also to express their gratitude to the useful suggestions of the Referee, which led to a substantial improvement of some of the results presented in the paper.
\end{ack}

\section{Preliminaries and notation}\label{preliminaries}
We start by fixing some notation and recalling the definition of Kodaira dimension in the almost complex category.

Let $M$ be a compact $2n$-dimensional smooth manifold endowed with an almost complex structure $J$. Denote by $N_J$ the {\em Nijenhuis tensor} of $J$, that is, the smooth $(1,1)$-tensor field defined by
\[N_J(X,Y)=[JX,JY]-[X,Y]-J[JX,Y]-J[X,JY]\]
for every pair of vector fields $X$, $Y$ on $M$. The Newlander--Nirenberg Theorem states that $J$ is integrable if and only if $N_J$ vanishes.

Following \cite{Chen-Zhang-1}, we briefly recall the definition of \emph{Kodaira dimension} of $(M, J)$.

Denote by $\Lambda_J^{p, q} M$ the bundle of smooth $(p, q)$-forms on $(M, J)$ and by $A_J^{p, q}(M) = \Gamma\pa{M, \Lambda_J^{p, q} M}$ the space of smooth $(p, q)$-forms on $(M, J)$. Then the exterior differential $d$ satisfies
\[d(A_J^{p, q}(M)) \subset A_J^{p + 2, q - 1}(M) + A_J^{p + 1, q}(M) + A_J^{p, q + 1}(M) + A_J^{p-1, q + 2}(M),\]
and, consequently, $d$ decomposes as
\[d = \mu_J + \partial_J + \overline{\partial}_J + \overline{\mu}_J,\]
where $\mu_J = \pi^{p + 2, q - 1} \circ d$ and $\overline{\partial}_J = \pi^{p, q + 1} \circ d$. 
Let $\cK_X = \Lambda_J^{n, 0} M$ be the \emph{canonical bundle} of the almost complex manifold $X = (M, J)$. Then $\cK_X$ is a complex line bundle over $X$ and the $\overline{\partial}_J$-operator on $(M, J)$ gives rise to a pseudoholomorphic structure on $\cK_X$, i.e., a differential operator (still denoted by $\overline{\partial}_J$)
\[\overline{\partial}_J: \Gamma(M, \cK_X) \to \Gamma(M, \Lambda^{0, 1}_J M \otimes \cK_X)\]
satisfying the Leibniz rule
\[\overline{\partial}_J (f \sigma) = \overline{\partial}_J f \otimes \sigma + f \overline{\partial}_J \sigma,\]
for every smooth function $f$ and section $\sigma$.

As a consequence of Hodge Theory (see \cite[Theorem 1.1]{Chen-Zhang-1}), the space $H^0(M,\cK_X^{\otimes m})$ of pseudoholomorphic pluricanonical sections is a finite dimensional complex vector space for every $m \geq 1$.

\begin{defin*}[{\cite[Definition 1.2]{Chen-Zhang-1}}]
The \emph{$m^{\text{th}}$-plurigenus} of $(M, J)$ is
\[P_m(M, J) := \dim_\C H^0(M,\cK_X^{\otimes m}).\]
The \emph{Kodaira dimension} of $(M, J)$ is defined as
\[\kod(M, J) := \left\{
\begin{array}{ll}
-\infty & \text{if } P_m(M, J) = 0 \text{ for every } m \geq 1,\\
\displaystyle \limsup_{m \to +\infty} \frac{\log P_m(M, J)}{\log m} & \text{otherwise}.
\end{array}
\right.\]
\end{defin*}

\section{\texorpdfstring{$4$}{4}-solvmanifolds without complex structures}\label{almost-complex-solvmanifolds}

\subsection{The solvmanifolds \texorpdfstring{$M(k)$}{M(k)}}\label{sect: M(k)}

Based on the classification of the pairs consisting of a connected, simply connected, (real) three dimensional solvable Lie group and a discrete cocompact lattice made in \cite{Auslander-Green-Hahn}, Fern\'andez and Gray produced in \cite{Fernandez-Gray} an example of compact four dimensional manifold admitting no integrable almost complex structures. In this section we review the construction of this manifold, with the purpose of introducing the notation.

Fix $k \in \IR \smallsetminus \set{0}$ such that $e^k + e^{-k} \in \IZ$, and consider the group
\[G(k) = \set{\elm{x}{y}{z} \st x, y, z \in \IR}.\]
Then we can see $G(k) \simeq \IR^2 \rtimes_{\phi_k} \IR$, where the semidirect product is taken with respect to the action of $\IR$ on $\IR^2$ given by
\[\begin{array}{rccl}
\phi_k: & \IR & \longrightarrow & \GL(2, \IR)\\
 & z & \longmapsto & \left( \begin{array}{cc}
e^{kz} & 0\\
0 & e^{-kz}
\end{array} \right).
\end{array}\]
It is then easy to show that $G(k)$ is a solvable, non nilpotent, Lie group.

Fix now two linearly independent vectors $u = (u_1, u_2), v = (v_1, v_2) \in \IR^2$ and an integer $n \in \IZ$, and let $D = D(u, v, n)$ be the subgroup of $G(k)$ generated by
\[\begin{array}{c}
E_u = \left( \begin{array}{cccc}
1 & 0 & 0 & u_1\\
0 & 1 & 0 & u_2\\
0 & 0 & 1 & 0\\
0 & 0 & 0 & 1 \end{array} \right), \qquad E_v = \left( \begin{array}{cccc}
1 & 0 & 0 & v_1\\
0 & 1 & 0 & v_2\\
0 & 0 & 1 & 0\\
0 & 0 & 0 & 1
\end{array} \right),\\
E_n = \left( \begin{array}{cccc}
e^{nk} & 0 & 0 & 0\\
0 & e^{-nk} & 0 & 0\\
0 & 0 & 1 & n\\
0 & 0 & 0 & 1
\end{array} \right).
\end{array}\]
Then $D$ is a discrete cocompact lattice in $G(k)$ (see \cite[Theorem 5(4)]{Auslander-Green-Hahn}).

On the product $G(k) \times \IR$ we have the action of $D \times \IZ$, where $D$ acts on $G(k)$ by multiplication \emph{on the left}, while $\IZ$ acts on $\IR$ by translations. Let $M(k)$ be the quotient of $G(k) \times \IR$ by this diagonal action, then $M(k)$ is a compact manifold, admitting the following $1$-forms:
\[\begin{array}{cl}
e^1 & \text{induced by the invariant form } e^{-kz} dx,\\
e^2 & \text{induced by the invariant form } e^{kz} dy,\\
e^3 & \text{induced by the invariant form } dz,\\
e^4 & \text{induced by the invariant form } dt.
\end{array}\]
Dually, the corresponding fields are
\begin{equation}\label{eq: fields M(k)}
\begin{array}{cl}
e_1 & \text{induced by the invariant field } e^{kz} \frac{\partial}{\partial x},\\
e_2 & \text{induced by the invariant field } e^{-kz} \frac{\partial}{\partial y},\\
e_3 & \text{induced by the invariant field } \frac{\partial}{\partial z},\\
e_4 & \text{induced by the invariant field } \frac{\partial}{\partial t}.
\end{array}
\end{equation}
The structure equations for $M(k)$ and the only non trivial commutators are
\[de^1 = k e^1 \wedge e^3, \qquad de^2 = -k e^2 \wedge e^3, \qquad de^3 = 0, \qquad de^4 = 0\]
and
\[[e_1, e_3] = -k e_1, \qquad [e_2, e_3] = k e_2\]
respectively.

We will consider on $M(k)$ the following additional structures:
\begin{itemize}
\item the $2$-form
\[\omega = e^1 \wedge e^2 + e^3 \wedge e^4,\]
which is easily seen to be a symplectic form on $M(k)$;
\item the Riemannian metric
\[g = e^1 \tensor e^1 + e^2 \tensor e^2 + e^3 \tensor e^3 + e^4 \tensor e^4.\]
\end{itemize}

\begin{rem}
The manifold $M(k)$ is (real) parallelizable, since the vectors $e_1(P)$, $e_2(P)$, $e_3(P)$ and $e_4(P)$ span the tangent space $T_P M(k)$ for every $P \in M(k)$. As a consequence, we can define an almost complex structure on $M(k)$ by simply specifying its action on $H^0(M(k), TM(k))$.
\end{rem}

\begin{rem}
Observe that a smooth function $f: M(k) \longrightarrow \IR$ can be identified with a smooth function (still denoted by $f$) on the universal cover $\IR^4$ such that
\[f(x, y, z, t) = f(e^{kn\gamma}x + \alpha u_1 + \beta v_1, e^{-kn\gamma}y + \alpha u_2 + \beta v_2, z + n\gamma, t + \varepsilon)\]
for every $\alpha, \beta, \gamma, \varepsilon \in \IZ$. In particular, for $\gamma = 0$ we see that $f$ is $\IZ^3$-periodic, and so it admits a Fourier series expansion of the form
\begin{equation}\label{eq: Fourier M(k)}
f(x, y, z, t) = \sum_{I = (a, b, c) \in \IZ^3} f_I(z) e^{2\pi \ii \frac{1}{\delta} ((a v_2 - b u_2)x + (-a v_1 + b u_1)y + c \delta t)},
\end{equation}
where $\delta = u_1 v_2 - u_2 v_1 \neq 0$.
\end{rem}

\subsection{The nilmanifold \texorpdfstring{$\cN$}{N}}

Let $\gothG$ be the group $\pa{\IR^4, *}$, where the product of $\pa{a, b, c, d}, \pa{x, y, z, t} \in \IR^4$ is defined as
\[\pa{a, b, c, d} * \pa{x, y, z, t} = \pa{x + a, y + b, z + ay + c, t + \frac{1}{2} a^2 y + az + d}.\]

It is then straightforward to verify that the set of $1$-forms
\[\begin{array}{l}
e^1 = dx,\\
e^2 = dy,\\
e^3 = dz - x dy,\\
e^4 = dt + \frac{1}{2} x^2 dy - x dz
\end{array}\]
on $\gothG$ is a basis for the space $\Gamma\pa{\gothG, T \gothG}$, and that these forms are left invariant with respect to translations by elements in the subgroup consisting of $4$-tuples with \emph{even} integral entries.

Call $\cN$ the quotient of $\gothG$ by such an action, then the previous forms descend to $\cN$ and their valuations at any given point span the cotangent space at that point. The dual tangent frame to $\cN$ is given by the fields
\begin{equation}\label{eq: tangent fields N}
\begin{array}{l}
e_1 = \frac{\partial}{\partial x},\\
e_2 = \frac{\partial}{\partial y} + x \frac{\partial}{\partial z} + \frac{1}{2} x^2 \frac{\partial}{\partial t},\\
e_3 = \frac{\partial}{\partial z} + x \frac{\partial}{\partial t},\\
e_4 = \frac{\partial}{\partial t}.
\end{array}
\end{equation}

The structure equations for $\cN$ and the only non trivial commutators are
\begin{equation}\label{eq: structure equations gothM}
de^1 = 0, \qquad de^2 = 0, \qquad de^3 = -e^1 \wedge e^2, \qquad de^4 = -e^1 \wedge e^3,
\end{equation}
and
\[[e_1, e_2] = e_3, \qquad [e_1, e_3] = e_4,\]
respectively.

\begin{rem}
The $2$-form
\[\xi = e^1 \wedge e^4 + e^2 \wedge e^3\]
is a symplectic form on $\cN$. Indeed, $\xi \wedge \xi = 2 e^1 \wedge e^2 \wedge e^3 \wedge e^4$ and by the structure equations $d\xi = 0$.
\end{rem}

\begin{rem}
As $\cN$ is a quotient of $\IR^4$ by the action of a suitable lattice, we can see a smooth function $f: \cN \longrightarrow \IR$ as an invariant function on $\IR^4$, where invariance is meant in terms of the lattice:
\begin{equation}\label{eq: periodicity N}
f(x, y, z, t) = f(x + 2\alpha, y + 2\beta, z + 2\alpha y + 2\gamma, t + 2\alpha^2 y + 2 \alpha z + 2\delta) \qquad \forall \, \alpha, \beta, \gamma, \delta \in \IZ.
\end{equation}
In particular $f$ is periodic of period $2$ in $y, z, t$, hence it admits a Fourier expansion of the form
\begin{equation}\label{eq: Fourier N}
f(x, y, z, t) = \sum_{I = (a, b, c) \in \IZ^3} f_I(x) e^{\ii \pi (ay + bz + ct)}.
\end{equation}
\end{rem}

\subsection{The solvmanifolds \texorpdfstring{$\cM(\lambda)$}{M(l)}}\label{sect: M(lambda)}

Let $\lambda \in \left( -1, -\frac{1}{2} \right)$ and define $\cG(\lambda)$ as the group
\[\cG(\lambda) = \set{\left( \begin{array}{ccccc}
e^{\lambda t} & 0 & 0 & 0 & x\\
0 & e^t & 0 & 0 & y\\
0 & 0 & e^{-(1 + \lambda) t} & 0 & z\\
0 & 0 & 0 & 1 & t\\
0 & 0 & 0 & 0 & 1
\end{array} \right) \st x, y, z, t \in \IR}.\]

Then $\cG(\lambda)$ is a simply connected solvable Lie group isomorphic to
\[\IR^3 \rtimes_\varphi \IR, \text{ where } \varphi(t) = \left( \begin{array}{ccc}
e^{\lambda t} & 0 & 0\\
0 & e^t & 0\\
0 & 0 & e^{-(1 + \lambda)t}
\end{array} \right)\]
and we can characterize the values of the parameter $\lambda$ for which it admits a lattice. In fact, by \cite[Proposition 2.1]{Lee-Lee-Shin-Yi} we have that $\cG(\lambda)$ admits a lattice if and only if there exist integers $m, n$ with $m \neq n$ such that the equation $x^3 - m x^2 + n x - 1 = 0$ has $3$ real, positive and distinct solutions $\alpha_1 > \alpha_2 > \alpha_3$. Assume this is the case, and define
\[(\lambda, t_0) = \left\{ \begin{array}{ll}
\left( \frac{\log \alpha_3}{\log \alpha_1}, \log \alpha_1 \right), & \text{if } \alpha_1 > 1 > \alpha_2 > \alpha_3,\\
\left( \frac{\log \alpha_1}{\log \alpha_3}, \log \alpha_3 \right), & \text{if } \alpha_1 > \alpha_2 > 1> \alpha_3;
\end{array} \right.\]
then there exists a matrix $A \in \operatorname{GL}(3, \IZ)$ which is similar to the matrix $\varphi(t_0)$. Let $P \in \operatorname{GL}(3, \IR)$ be such that $\varphi(t_0) = P \cdot A \cdot P^{-1}$. Let $\IZ$ act on $\IZ^3$ via the matrix $A$ and consider the semidirect product $\IZ^3 \rtimes \IZ$, it is possible to embed this group in $\cG(\lambda)$ as a lattice in the following way: we send the canonical generators of $\IZ^3$ to the vectors in $\IR^3$ corresponding to the columns of the matrix $P$ and the generator of $\IZ$ corresponding to $A$ to $t_0$.

We let $\cM(\lambda)$ be the quotient of $\cG(\lambda)$ by the \emph{left} action of translations by elements of a lattice in $\cG(\lambda)$ (being understood that $\lambda$ is chosen in such a way that a lattice exists).

The $1$-forms
\[\begin{array}{l}
e^1 = dt,\\
e^2 = e^{-t} dy,\\
e^3 = e^{-\lambda t} dx,\\
e^4 = e^{(1 + \lambda) t} dz,
\end{array}\]
generate the cotangent space at each point of $\cM(\lambda)$, and it is easy to see that they satisfy the structure equations
\begin{equation}\label{eq: structure equations cm(lambda)}
de^1 = 0, \qquad de^2 = -e^1 \wedge e^2, \qquad de^3 = -\lambda e^1 \wedge e^3, \qquad de^4 = (1 + \lambda) e^1 \wedge e^4.
\end{equation}

The dual picture on the tangent bundle is as follows: the dual tangent frame is
\begin{equation}\label{eq: tangent vectors cm(lambda)}
\begin{array}{l}
e_1 = \frac{\partial}{\partial t},\\
e_2 = e^t \frac{\partial}{\partial y},\\
e_3 = e^{\lambda t} \frac{\partial}{\partial x},\\
e_4 = e^{-(1 + \lambda) t} \frac{\partial}{\partial z},
\end{array}
\end{equation}
and the non trivial commutators are
\[[e_1, e_2] = e_2, \qquad [e_1, e_3] = \lambda e_3, \qquad [e_1, e_4] = -(1 + \lambda) e_4.\]

Hence $\cM(\lambda)$ is a $2$-step (non nilpotent) solvmanifold. It is known (cf.~\cite[Section 4]{Hasegawa}) that $\cM(\lambda)$ admits no integrable almost complex structures.

\begin{rem}
As $\cM(\lambda)$ is a quotient of $\IR^4$ by the action of a suitable lattice $\Gamma$, we can see a smooth function $f: \cM(\lambda) \longrightarrow \IR$ as an invariant function on $\IR^4$, where invariance is meant in terms of the lattice:
\begin{equation}\label{eq: periodicity M(lambda)}
f(x, y, z, t) = f(e^{\lambda \delta}x + \alpha, e^\delta y + \beta, e^{-(1 + \lambda) \delta}z + \gamma, t + \delta) \qquad \forall \, (\alpha, \beta, \gamma, \delta) \in \Gamma.
\end{equation}
In particular, $f$ is periodic in $(x, y, z)$ with periods given by the columns of the matrix $P$, hence $f$ admits a Fourier expansion of the form
\begin{equation}\label{eq: Fourier M(lambda)}
f(x, y, z, t) = \sum_{I = (a, b, c) \in \IZ^3} f_I(t) e^{2 \pi \ii (\lambda_I x + \mu_I y + \nu_I z)},
\end{equation}
with $(\lambda_I, \mu_I, \nu_I) = I \cdot P^{-1} = (a, b, c) \cdot P^{-1}$.
\end{rem}

\section{Special families of almost complex structures on \texorpdfstring{$M(k)$}{M(k)}}\label{s5}

Consider the following families of endomorphisms of the tangent bundle of $M(k)$:
\begin{equation}\label{eq: first j}
\begin{array}{l}
J_s = \left( \begin{array}{cccc}
\alpha_1 & \beta_1 & 0 & 0\\
\gamma_1 & -\alpha_1 & 0 & 0\\
0 & 0 & \alpha_2 & \beta_2\\
0 & 0 & \gamma_2 & -\alpha_2
\end{array} \right),\\ \\
\text{with } \left\{ \begin{array}{l}
s = (r_1, s_1, r_2, s_2) \in \IR^4,\\
\alpha_i = -\frac{2s_i}{r_i^2 + s_i^2 - 1},\\
\beta_i = \frac{r_i^2 + 2 r_i + s_i^2 + 1}{r_i^2 + s_i^2 - 1},\\
\gamma_i = -\frac{r_i^2 - 2 r_i + s_i^2 + 1}{r_i^2 + s_i^2 - 1},
\end{array} \right.
\end{array}
\end{equation}
and
\[\begin{array}{l}
J_r = \alpha \left( \begin{array}{cccc}
0 & -(1 - r)^2 & 2r^2 & -2r(1 - r)\\
(1 - r)^2 & 0 & -2r(1 - r) & -2r^2\\
-2r^2 & 2r(1 - r) & 0 & -(1 - r)^2\\
2r(1 - r) & 2r^2 & (1 - r)^2 & 0
\end{array} \right),\\ \\
\text{with } \left\{ \begin{array}{l}
r \in \IR,\\
\alpha = \frac{1}{(1 - r)^2 + 2r^2} = \frac{1}{3r^2 - 2r + 1}.
\end{array} \right.
\end{array}\]
It is then easy to see that $\alpha_i^2 + \beta_i \gamma_i = -1$ for $i = 1, 2$ and that both $J_s$ and $J_r$ define smooth families of almost complex structures on $M(k)$, as $J_s^2 = J_r^2 = -\id_{TM(k)}$. We recall here that since $M(k)$ can not admit an integrable almost complex structure, none of the $J_s$'s nor of the $J_r$'s is integrable.

The almost complex structures $J_s$ behave well with respect to the symplectic form $\omega = e^1 \wedge e^2 + e^3 \wedge e^4$, at least for small values of the parameter $s = (r_1, s_1, r_2, s_2)$: we have in fact that $\omega(J_s \blank, J_s \blank) = \omega(\blank, \blank)$, and that the symmetric bilinear form $g_s(\blank, \blank) = \omega(\blank, J_s \blank)$ is positive definite (hence it is a Riemannian metric on $M(k)$). The matrix representation for the metric $g_s$ in the basis $\set{e_1, e_2, e_3, e_4}$ is
\[g_s = \left( \begin{array}{cccc}
\gamma_1 & -\alpha_1 & 0 & 0\\
-\alpha_1 & -\beta_1 & 0 & 0\\
0 & 0 & \gamma_2 & -\alpha_2\\
0 & 0 & -\alpha_2 & -\beta_2
\end{array} \right).\]
Since the fundamental form of all the almost Hermitian manifolds $(M(k), g_s, J_s)$ is $\omega$, we have a family of almost K\"ahler manifolds.

On the other hand, the almost complex structures $J_r$ are all compatible with the fixed metric $g$, as $g(J_r \blank, J_r \blank) = g(\blank, \blank)$, hence we have an almost Hermitian manifold $(M(k), g, J_r)$ for all $r \in \IR$. By a direct computation, it is easy to see that the fundamental form $\omega_r(\blank, \blank) = g(J_r \blank, \blank)$ is $d$-closed if and only if $r = 0$. So, apart for $r = 0$, none of the manifolds in this family is almost K\"ahler.

We remark that both families $\{J_s\}$ and $\{J_r\}$ are deformations of the almost complex structure $J_0$ defined by $s = 0$ in the first family and by $r = 0$ in the second family.

\subsection{Kodaira dimension}

We focus first on the family $(M(k), J_s)$.

We can then find a $g_s$-orthonormal frame $\set{\varepsilon_1(s), \varepsilon_2(s), \varepsilon_3(s), \varepsilon_4(s)}$ for the tangent bundle (to simplify the notation, from now on we will drop the explicit dependence on $s$):
\[\begin{array}{ll}
\varepsilon_1 = \frac{1}{\sqrt{\gamma_1}} e_1, & \varepsilon_3 = \frac{1}{\sqrt{\gamma_2}} e_3,\\
\varepsilon_2 = \frac{1}{\sqrt{\gamma_1}} (\alpha_1 e_1 + \gamma_1 e_2), & \varepsilon_4 = \frac{1}{\sqrt{\gamma_2}} (\alpha_2 e_3 + \gamma_2 e_4),
\end{array}\]
whose corresponding dual frame is
\[\begin{array}{ll}
\varepsilon^1 = \sqrt{\gamma_1} \pa{e^1 - \frac{\alpha_1}{\gamma_1} e^2}, & \varepsilon^3 = \sqrt{\gamma_2} \pa{e^3 - \frac{\alpha_2}{\gamma_2} e^4},\\
\varepsilon^2 = \frac{1}{\sqrt{\gamma_1}} e^2, & \varepsilon^4 = \frac{1}{\sqrt{\gamma_2}} e^4.
\end{array}\]

The commutators of the fields and, dually, the differentials of these forms are respectively given by
\[\begin{array}{lll}
[\varepsilon_1, \varepsilon_2] = 0, & [\varepsilon_1, \varepsilon_3] = -\frac{k}{\sqrt{\gamma_2}} \varepsilon_1, & [\varepsilon_1, \varepsilon_4] = -\frac{k \alpha_2}{\sqrt{\gamma_2}} \varepsilon_1,\\
{[}\varepsilon_2, \varepsilon_3] = -2 \frac{k \alpha_1}{\sqrt{\gamma_2}} \varepsilon_1 + \frac{k}{\sqrt{\gamma_2}} \varepsilon_2, & [\varepsilon_2, \varepsilon_4] = -2 \frac{k \alpha_1 \alpha_2}{\sqrt{\gamma_2}} \varepsilon_1 + \frac{k \alpha_2}{\sqrt{\gamma_2}} \varepsilon_2, & [\varepsilon_3, \varepsilon_4] = 0,
\end{array}\]
and by
\[\begin{array}{ll}
d\varepsilon^1 = \frac{k}{\sqrt{\gamma_2}} \varepsilon^1 \wedge \varepsilon^3 + \frac{k \alpha_2}{\sqrt{\gamma_2}} \varepsilon^1 \wedge \varepsilon^4 + 2 \frac{k \alpha_1}{\sqrt{\gamma_2}} \varepsilon^2 \wedge \varepsilon^3 + 2 \frac{k \alpha_1 \alpha_2}{\sqrt{\gamma_2}} \varepsilon^2 \wedge \varepsilon^4, & d\varepsilon^3 = 0,\\
d\varepsilon^2 = -\frac{k}{\sqrt{\gamma_2}} \varepsilon^2 \wedge \varepsilon^3 - \frac{k \alpha_2}{\sqrt{\gamma_2}} \varepsilon^2 \wedge \varepsilon^4, & d\varepsilon^4 = 0.
\end{array}\]

Now, we switch from the real to the complex formalism. Let $h_s$ be the Hermitian metric induced by $g_s$, we consider the $h_s$-unitary frame of fields of type $(1, 0)$
\[\cX_1 = \frac{\sqrt{2}}{2} \pa{\varepsilon_1 - \ii \varepsilon_2}, \qquad \cX_2 = \frac{\sqrt{2}}{2} \pa{\varepsilon_3 - \ii \varepsilon_4},\]
and its dual frame of $(1, 0)$-forms
\[\varphi^1 = \frac{\sqrt{2}}{2} \pa{\varepsilon^1 + \ii \varepsilon^2}, \qquad \varphi^2 = \frac{\sqrt{2}}{2} \pa{\varepsilon^3 + \ii \varepsilon^4}.\]
It is in fact easy to see that the almost complex structure $J_s$ takes the standard form on the basis $\set{\varepsilon_1(s), \varepsilon_2(s), \varepsilon_3(s), \varepsilon_4(s)}$. It follows that
\[\begin{array}{rl}
d\varphi^1 = & \frac{\sqrt{2} k}{2 \sqrt{\gamma_2}} (-\ii \alpha_1 (1 - \ii \alpha_2) \varphi^1 \wedge \varphi^2 +\\
 & - \ii \alpha_1 (1 + \ii \alpha_2) \varphi^1 \wedge \bar{\varphi}^2 +\\
 & - (1 + \ii \alpha_1)(1 - \ii \alpha_2) \varphi^2 \wedge \bar{\varphi}^1 +\\
 & + (1 + \ii \alpha_1)(1 + \ii \alpha_2) \bar{\varphi}^1 \wedge \bar{\varphi}^2),\\
d\varphi^2 = & 0,
\end{array}\]
from which we deduce that
\[\begin{array}{l}
\delbar \varphi^1 = -\frac{\sqrt{2} k}{2 \sqrt{\gamma_2}} \left( \ii \alpha_1 (1 + \ii \alpha_2) \varphi^1 \wedge \bar{\varphi}^2 + (1 + \ii \alpha_1)(1 - \ii \alpha_2) \varphi^2 \wedge \bar{\varphi}^1 \right),\\
\delbar \varphi^2 = 0.
\end{array}\]

We observe that $\varphi^1 \wedge \varphi^2$ is a smooth section of the canonical bundle of $(M(k), J_s)$ and
\[\delbar(\varphi^1 \wedge \varphi^2) = \frac{\sqrt{2} k}{2 \sqrt{\gamma_2}} \ii \alpha_1(1 + \ii \alpha_2) \varphi^1 \wedge \varphi^2 \wedge \bar{\varphi}^2.\]

\begin{prop}\label{prop: kod dim m(k) 1}
In a suitable neighbourhood of the origin the Kodaira dimension of $(M(k), J_s)$ is
\[\kod(M(k), J_s) = \left\{ \begin{array}{ll}
-\infty & \text{if } \alpha_1 \neq 0,\\
0 & \text{if } \alpha_1 = 0.
\end{array} \right.\]
\end{prop}

\begin{proof}
Let $f \varphi^1 \wedge \varphi^2$ be a smooth section of the canonical bundle. This section is pseudoholomorphic if and only if
\[0 = \delbar(f \varphi^1 \wedge \varphi^2) = \bar{\cX}_1(f) \varphi^1 \wedge \varphi^2 \wedge \bar{\varphi}^1 + \bar{\cX}_2(f) \varphi^1 \wedge \varphi^2 \wedge \bar{\varphi}^2 + f \delbar(\varphi^1 \wedge \varphi^2),\]
hence if and only if
\begin{equation}\label{eq: system P1}
\left\{ \begin{array}{l}
\bar{\cX}_1(f) = 0\\
\bar{\cX}_2(f) + \frac{\sqrt{2} k}{2 \sqrt{\gamma_2}} \ii f \alpha_1 (1 + \ii \alpha_2) = 0.
\end{array} \right.
\end{equation}
Write $f = u + \ii v$, where we see $u$ and $v$ as real functions of the \emph{real} variables $x$, $y$, $z$ and $t$, which are defined on $\IR^4$ and are periodic with respect to the action of the lattice $D \times \IZ$.

The first equation in \eqref{eq: system P1} then becomes
\[0 = \bar{\cX}_1(f) = \frac{\sqrt{2}}{2}(\varepsilon_1 + \ii \varepsilon_2)(u + \ii v) = \frac{\sqrt{2}}{2}(\varepsilon_1(u) - \varepsilon_2(v) + \ii (\varepsilon_1(v) + \varepsilon_2(u))),\]
which leads us to the system
\[\left\{ \begin{array}{l}
e^{kz} \frac{\partial u}{\partial x} - \alpha_1 e^{kz} \frac{\partial v}{\partial x} - \gamma_1 e^{-kz} \frac{\partial v}{\partial y} = 0\\
e^{kz} \frac{\partial v}{\partial x} + \alpha_1 e^{kz} \frac{\partial u}{\partial x} + \gamma_1 e^{-kz} \frac{\partial u}{\partial y} = 0.
\end{array} \right.\]
We can then express $\frac{\partial u}{\partial x}$ and $\frac{\partial u}{\partial y}$ in terms of $\frac{\partial v}{\partial x}$ and $\frac{\partial v}{\partial y}$:
\[\left\{ \begin{array}{l}
\frac{\partial u}{\partial x} = \alpha_1 \frac{\partial v}{\partial x} + \gamma_1 e^{-2kz} \frac{\partial v}{\partial y}\\
\frac{\partial u}{\partial y} = e^{2kz} \beta_1 \frac{\partial v}{\partial x} - \alpha_1 \frac{\partial v}{\partial y},
\end{array} \right.\]
and so once we take the derivative with respect to $y$ of the first relation and with respect to $x$ of the second one, we see that the relation
\[\alpha_1 \frac{\partial^2 v}{\partial x \partial y} + \gamma_1 e^{-2kz} \frac{\partial^2 v}{\partial y^2} = e^{2kz} \beta_1 \frac{\partial^2 v}{\partial x^2} - \alpha_1 \frac{\partial^2 v}{\partial x \partial y}\]
must hold. Hence $v$ is a solution of the elliptic differential equation
\[\pa{-e^{2kz} \beta_1 \frac{\partial^2}{\partial x^2} + 2\alpha_1 \frac{\partial^2}{\partial x \partial y} + e^{-2kz} \gamma_1 \frac{\partial^2}{\partial y^2}}v = 0,\]
and so it must be constant with respect to $x$ and $y$ (because of the periodicity). An analogous argument shows that also $u$ must be constant with respect to $x$ and $y$.

Consider the Fourier series expansion \eqref{eq: Fourier M(k)} of $u$ and $v$. Since we know that they do not depend on $x$ and $y$, this expansion can be simplified further to
\begin{equation}\label{eq: fourier simpl}
u(x, y, z, t) = \sum_{I = (a, b) \in \IZ^2} u_I e^{2\pi \ii \pa{\frac{a}{n} z + bt}}, \qquad u_I \in \IC
\end{equation}
and similarly for $v$.

The second equation in \eqref{eq: system P1} is equivalent to the system
\[\left\{ \begin{array}{l}
\frac{\partial u}{\partial z} - \alpha_2 \frac{\partial v}{\partial z} - \gamma_2 \frac{\partial v}{\partial t} = k \alpha_1 (\alpha_2 u + v)\\
\frac{\partial v}{\partial z} + \alpha_2 \frac{\partial u}{\partial z} + \gamma_2 \frac{\partial u}{\partial t} = k \alpha_1 (\alpha_2 v - u),
\end{array} \right.\]
in view of \eqref{eq: fourier simpl} we obtain the following system for $(u_I, v_I)$
\[\left\{ \begin{array}{l}
2\pi \ii \frac{a}{n} u_I - 2\pi \ii \frac{a}{n} \alpha_2 v_I - 2\pi \ii b \gamma_2 v_I - k\alpha_1 \alpha_2 u_I - k\alpha_1 v_I = 0\\
2\pi \ii \frac{a}{n} v_I + 2\pi \ii \frac{a}{n} \alpha_2 u_I + 2\pi \ii b \gamma_2 u_I - k\alpha_1 \alpha_2 v_I + k\alpha_1 u_I = 0.
\end{array} \right.\]
This is a homogeneous linear system, whose representing matrix has determinant
\begin{equation}\label{eq: det1}
-4\pi^2 \frac{a^2}{n^2} + k^2 \alpha_1^2(\alpha_2^2 + 1) - 4 \pi^2 \pa{\frac{a}{n}\alpha_2 + b\gamma_2}^2 + 4\pi \ii b k \alpha_1 \gamma_2.
\end{equation}
We want to determine when this determinant vanishes.
\begin{enumerate}
\item If $\alpha_1 = 0$, the determinant \eqref{eq: det1} vanishes if and only if $a = b = 0$.\\
In fact for $\alpha_1 = 0$ the imaginary part of \eqref{eq: det1} is zero and its real part is $-4\pi^2 \pa{\frac{a^2}{n^2} + \pa{\frac{a}{n} \alpha_2 + b \gamma_2}^2}$. So this determinant vanishes if and only if $a = b\gamma_2 = 0$, which happens if and only if $a = b = 0$ in a suitable neighbourhood of the origin.
\item If $\alpha_1 \neq 0$, the determinant \eqref{eq: det1} is always non-zero.\\
In fact, if $\alpha_1 \neq 0$ then in a neighbourhood of the origin the imaginary part of \eqref{eq: det1} vanishes if and only if $b = 0$. In this case the real part of \ref{eq: det1} is
\[(1 + \alpha_2^2)\pa{k \alpha_1 + 2\pi \frac{a}{n}}\pa{k \alpha_1 - 2\pi \frac{a}{n}},\]
which is zero if and only if $a = \pm \frac{k \alpha_1 n}{2 \pi}$. In a suitable neighbourhood of the origin we have $\abs{\frac{k \alpha_1 n}{2 \pi}} < 1$, so on such a neighbourhood the only possibility is $a = 0$. But then this implies that $\alpha_1 = 0$, which is not the case.
\end{enumerate}

We can summarize this result saying that $f$ is a solution of \eqref{eq: system P1} if and only if
\[\left\{ \begin{array}{ll}
f = \const & \text{if } \alpha_1 = 0,\\
f = 0 & \text{if } \alpha_1 \neq 0.
\end{array} \right.\]

To conclude the computation of the Kodaira dimension of $(M(k), J_s)$, we observe that it is not difficult to see that
\[\delbar((\varphi^1 \wedge \varphi^2)^{\tensor m}) = \frac{\sqrt{2} k}{2 \sqrt{\gamma_2}} \ii \alpha_1 (1 + \ii \alpha_2) m \bar{\varphi}^2 \otimes (\varphi^1 \wedge \varphi^2)^{\tensor m}\]
as a section of ${T^*}^{0, 1} M(k) \tensor \pa{\wedge^{2, 0} T^* M(k)}^{\tensor m}$. Hence $\delbar(f \cdot (\varphi^1 \wedge \varphi^2)^{\tensor m}) = 0$ leads us to a situation which is completely analogous to the one described in \eqref{eq: system P1}. As a consequence, we can claim that
\[P_m(M(k), J_s) = \left\{ \begin{array}{ll}
0 & \text{if } \alpha_1 \neq 0,\\
1 & \text{if } \alpha_1 = 0,
\end{array} \right.\]
and so
\[\kappa(M(k), J_s) = \left\{ \begin{array}{ll}
-\infty & \text{if } \alpha_1 \neq 0,\\
0 & \text{if } \alpha_1 = 0.
\end{array} \right.\]
\end{proof}

Consider now the family $(M(k), J_r)$. A $g_0$-orthonormal basis in which the almost complex structures $J_r$ are in canonical form is then given by
\[\begin{array}{ll}
\varepsilon_1 = \sqrt{\alpha}((1 - r) e_1 + r e_2 - r e_3), & \varepsilon_3 = \sqrt{\alpha}(r e_1 + (1 - r) e_3 + r e_4)\\
\varepsilon_2 = \sqrt{\alpha}(-r e_1 + (1 - r) e_2 + r e_4), & \varepsilon_4 = \sqrt{\alpha}(-r e_2 - r e_3 + (1 - r) e_4),
\end{array}\]
whose corresponding dual frame is
\[\begin{array}{ll}
\varepsilon^1 = \sqrt{\alpha}((1 - r) e^1 + r e^2 - r e^3), & \varepsilon^3 = \sqrt{\alpha}(r e^1 + (1 - r) e^3 + r e^4)\\
\varepsilon^2 = \sqrt{\alpha}(-r e^1 + (1 - r) e^2 + r e^4), & \varepsilon_4 = \sqrt{\alpha}(-r e^2 - r e^3 + (1 - r) e^4).
\end{array}\]

It follows that
\[\cX_1 = \frac{\sqrt{2}}{2}\pa{\varepsilon_1 - \ii \varepsilon_2}, \qquad \cX_2 = \frac{\sqrt{2}}{2}\pa{\varepsilon_3 - \ii \varepsilon_4}\]
is an $h$-unitary frame of fields of type $(1, 0)$, where $h$ is the Hermitian form associated to $g$ and $J_r$, and that
\[\varphi^1 = \frac{\sqrt{2}}{2}\pa{\varepsilon^1 + \ii \varepsilon^2}, \qquad \varphi^2 = \frac{\sqrt{2}}{2}\pa{\varepsilon^3 + \ii \varepsilon^4}\]
is the corresponding coframe of forms of type $(1, 0)$. A computation yields
\[\begin{array}{rl}
d\varphi^1 = & \frac{\sqrt{2}}{2} \alpha \sqrt{\alpha} k (1 - r - \ii r) (r^2 \varphi^1 \wedge \varphi^2 +\\
 & + r(1 - r - \ii r) \varphi^1 \wedge \bar{\varphi}^1 +\\
 & - ((1 - r)^2 + 2r^2) \varphi^2 \wedge \bar{\varphi}^1 +\\
 & + r(1 - r - \ii r) \varphi^2 \wedge \bar{\varphi}^2 +\\
 & + (1 - r - \ii r)^2 \bar{\varphi}^1 \wedge \bar{\varphi}^2),\\
d\varphi^2 = & \frac{\sqrt{2}}{2} \alpha \sqrt{\alpha} k r ((1 - r + \ii r)^2 \varphi^1 \wedge \varphi^2 +\\
 & - r(1 - r + \ii r) \varphi^1 \wedge \bar{\varphi}^1 +\\
 & + \frac{1}{\alpha} \varphi^1 \wedge \bar{\varphi}^2 +\\
 & - r(1 - r + \ii r) \varphi^2 \wedge \bar{\varphi}^2 +\\
 & + r^2 \bar{\varphi}^1 \wedge \bar{\varphi}^2.
\end{array}\]

We can then see that
\[\delbar(\varphi^1 \wedge \varphi^2) = \underbrace{-\frac{\sqrt{2}}{2} \alpha \sqrt{\alpha} k r ((1 - r - \ii r)^2 \bar{\varphi}^1 - r(1 - r + \ii r) \bar{\varphi}^2)}_{\tau} \wedge \varphi^1 \wedge \varphi^2.\]
As a consequence
\[\delbar\pa{(\varphi^1 \wedge \varphi^2)^{\tensor m}} = m\tau \tensor (\varphi^1 \wedge \varphi^2)^{\tensor m}, \qquad \text{for all } m \geq 1,\]
and so we see that the smooth pluricanonical section $(\varphi^1 \wedge \varphi^2)^{\tensor m}$ is pseudoholomorphic if and only if $r = 0$.

\begin{prop}\label{prop: kod dim m(k) 2}
We have the following:
\[P_m(M(k), J_r) = \left\{ \begin{array}{ll}
0 & \text{if } r = 1,\\
1 & \text{if } r = 0,
\end{array} \right.\]
and so
\[\kod(M(k), J_r) = \left\{ \begin{array}{ll}
-\infty & \text{if } r = 1,\\
0 & \text{if } r = 0.
\end{array} \right.\]
\end{prop}

\begin{proof}
This proposition is a particular case of Proposition \ref{prop: twistor on M(k)} (see also Remark \ref{rem: path in twistor}) for $r = 1$ and a particular case of Proposition \ref{prop: kod dim m(k) 1} for $r = 0$.
\end{proof}

\subsection{Curvature of the canonical connection}

The manifold $(M(k), J_s, g_s)$ is an almost K\"ahler manifold for every (small) value of the parameter $s$. As a consequence, we can use the same strategy used in \cite{Cattaneo-Nannicini-Tomassini} to compute the Chern--Ricci tensor and the scalar curvature of the canonical connection of these almost K\"ahler metrics.

Direct computations give:
\begin{lemma}
The Nijenhuis tensor of $J_s$ is the following:
\[\begin{array}{l}
N_{J_s}(\varepsilon_1, \varepsilon_2) = 0,\\
N_{J_s}(\varepsilon_1, \varepsilon_3) = 2 \frac{k}{\sqrt{\gamma_2}} (1 - \alpha_1 \alpha_2) \varepsilon_1 + 2 \frac{k}{\sqrt{\gamma_2}} (\alpha_1 + \alpha_2) \varepsilon_2,\\
N_{J_s}(\varepsilon_1, \varepsilon_4) = 2 \frac{k}{\sqrt{\gamma_2}} (\alpha_1 + \alpha_2) \varepsilon_1 - 2 \frac{k}{\sqrt{\gamma_2}} (1 - \alpha_1 \alpha_2) \varepsilon_2,\\
N_{J_s}(\varepsilon_2, \varepsilon_3) = 2 \frac{k}{\sqrt{\gamma_2}} (\alpha_1 + \alpha_2) \varepsilon_1 - 2 \frac{k}{\sqrt{\gamma_2}} (1 - \alpha_1 \alpha_2) \varepsilon_2,\\
N_{J_s}(\varepsilon_2, \varepsilon_4) = -2 \frac{k}{\sqrt{\gamma_2}} (1 - \alpha_1 \alpha_2) \varepsilon_1 - 2 \frac{k}{\sqrt{\gamma_2}} (\alpha_1 + \alpha_2) \varepsilon_2,\\
N_{J_s}(\varepsilon_3, \varepsilon_4) = 0.
\end{array}\]
\end{lemma}

By \cite[Corollary 3.8]{Cattaneo-Nannicini-Tomassini} we can then deduce the real torsion forms of the canonical connection:
\[\begin{array}{l}
\Theta_\IR^1 = \frac{k}{2 \sqrt{\gamma_2}} (1 - \alpha_1 \alpha_2)(\varepsilon^1 \wedge \varepsilon^3 - \varepsilon^2 \wedge \varepsilon^4) + \frac{1}{2 \sqrt{\gamma_2}} (\alpha_1 + \alpha_2)(\varepsilon^1 \wedge \varepsilon^4 + \varepsilon^2 \wedge \varepsilon^3),\\
\Theta_\IR^2 = \frac{k}{2 \sqrt{\gamma_2}} (\alpha_1 + \alpha_2)(\varepsilon^1 \wedge \varepsilon^3 - \varepsilon^2 \wedge \varepsilon^4) - \frac{1}{2 \sqrt{\gamma_2}} (1 - \alpha_1 \alpha_2)(\varepsilon^1 \wedge \varepsilon^4 + \varepsilon^2 \wedge \varepsilon^3),\\
\Theta_\IR^3 = 0,\\
\Theta_\IR^4 = 0.\\
\end{array}\]

The complex torsion forms are easy to compute from these ones:
\[\begin{array}{l}
\Theta^1 = \frac{\sqrt{2}}{2} \pa{\Theta_\IR^1 + \ii \Theta_\IR^2} = \frac{\sqrt{2} k}{2 \sqrt{\gamma_2}} (1 + \ii \alpha_1)(1 + \ii \alpha_2) \bar{\varphi}^1 \wedge \bar{\varphi}^2,\\
\Theta^2 = \frac{\sqrt{2}}{2} \pa{\Theta_\IR^3 + \ii \Theta_\IR^4} = 0.
\end{array}\]
To deduce the complex connection forms, we have to solve the system given by the first structure equations:
\[\left\{ \begin{array}{l}
d\varphi^1 + \vartheta^1_1 \wedge \varphi^1 + \vartheta^1_2 \wedge \varphi^2 = \Theta^1\\
d\varphi^2 + \vartheta^2_1 \wedge \varphi^1 + \vartheta^2_2 \wedge \varphi^2 = \Theta^2\\
\vartheta^1_1 + \overline{\vartheta^1_1} = \vartheta^2_1 + \overline{\vartheta^1_2} = \vartheta^2_2 + \overline{\vartheta^2_2} = 0,
\end{array} \right.\]
whose solution is
\[\begin{array}{l}
\vartheta^1_1 = -\frac{\sqrt{2} k}{2 \sqrt{\gamma_2}} \ii \alpha_1 (1 - \ii \alpha_2) \varphi^2 - \frac{\sqrt{2} k}{2 \sqrt{\gamma_2}} \ii \alpha_1 (1 + \ii \alpha_2) \bar{\varphi}^2,\\
\vartheta^1_2 = -\frac{\sqrt{2} k}{2 \sqrt{\gamma_2}} (1 + \ii \alpha_1)(1 + \ii \alpha_2) \bar{\varphi}^1,\\
\vartheta^2_1 = \frac{\sqrt{2} k}{2 \sqrt{\gamma_2}} (1 - \ii \alpha_1)(1 - \ii \alpha_2) \varphi^1,\\
\vartheta^2_2 = 0.
\end{array}\]
From the knowledge of the connection forms, we can compute the curvature forms $\psi^i_j$ by means of the second structure equations, and we see that
\[\begin{array}{rl}
\psi^1_1 = & \frac{k^2}{2 \gamma_2}(1 + \alpha_1^2)(1 + \alpha_2^2) \varphi^1 \wedge \bar{\varphi}^1,\\
\psi^2_1 = & -\frac{k^2}{\gamma_2} \ii \alpha_1 (1 - \ii \alpha_1)(1 - \ii \alpha_2)^2 \varphi^1 \wedge \varphi^2 +\\
 & - \frac{k^2}{\gamma_2} \ii \alpha_1 (1 - \ii \alpha_1)(1 + \alpha_2^2) \varphi^1 \wedge \bar{\varphi}^2 +\\
 & -\frac{k^2}{2 \gamma_2}(1 + \alpha_1^2)(1 - \ii \alpha_2)^2 \varphi^2 \wedge \bar{\varphi}^1 +\\
 & + \frac{k^2}{2 \gamma_2}(1 + \alpha_1^2)(1 + \alpha_2^2) \bar{\varphi}^1 \wedge \bar{\varphi}^2,\\
\psi^1_2 = & -\overline{\psi^2_1},\\
\psi^2_2 = & -\frac{k^2}{2 \gamma_2}(1 + \alpha_1^2)(1 + \alpha_2^2) \varphi^1 \wedge \bar{\varphi}^1.\\
\end{array}\]

We can state the following.
\begin{prop}\label{Chern-flat}
All the manifolds $(M(k), J_s, g_s)$ in our family are Chern--Ricci flat, hence they also have vanishing scalar curvature.
\end{prop}
\begin{proof}
To compute the component $R_{k\bar{l}}$ of the Chern--Ricci tensor we recall that $R^i_{jk\bar{l}}$ is the coefficient of $\varphi^k \wedge \bar{\varphi}^l$ in $\psi^i_j$ and that $R_{k\bar{l}} = \sum_i R^i_{ik\bar{l}}$. But the only non-vanishing coefficients among the $R^i_{jk\bar{l}}$ are
\[\begin{array}{ll}
R^1_{11\bar{1}} = \frac{k^2}{2 \gamma_2}(1 + \alpha_1^2)(1 + \alpha_2^2), & R^1_{21\bar{2}} = -\frac{k^2}{2 \gamma_2} (1 + \alpha_1^2)(1 + \ii \alpha_2)^2,\\
R^1_{22\bar{1}} = \frac{k^2}{\gamma_2} \ii \alpha_1 (1 + \ii \alpha_1)(1 + \alpha_2^2), & R^2_{11\bar{2}} = -\frac{k^2}{\gamma_2} \ii \alpha_1 (1 - \ii \alpha_1)(1 + \alpha_2^2),\\
R^2_{12\bar{1}} = -\frac{k^2}{2 \gamma_2} (1 + \alpha_1^2)(1 - \ii \alpha_2)^2, & R^2_{21\bar{1}} = -\frac{k^2}{2 \gamma_2}(1 + \alpha_1^2)(1 + \alpha_2^2),
\end{array}\]
hence it is now easy to see that $R_{k\bar{l}} = 0$ for all $k, l = 1, 2$.
\end{proof}

\section{Special almost complex structures on \texorpdfstring{$\cN$}{N}}\label{s6}

\subsection{Kodaira dimension}
Let $J$ be the almost complex structure on $\cN$ which acts on the tangent fields \eqref{eq: tangent fields N} as
\begin{equation}\label{almost-kaehler-nil}
  J e_1 = e_4, \qquad J e_2 = e_3, \qquad J e_3 = -e_2, \qquad J e_4 = -e_1.
\end{equation}
We can then define the fields of type $(1, 0)$
\begin{equation}\label{eq: (1,0) fields}
\cX_1 = \frac{1}{2} (e_1 - \ii e_4), \qquad \cX_2 = \frac{1}{2} (e_2 - \ii e_3,)
\end{equation}
and the corresponding dual $(1, 0)$-forms
\[\varphi^1 = e^1 + \ii e^4, \qquad \varphi^2 = e^2 + \ii e^3.\]

It is then a straightforward computation with the structure equations \eqref{eq: structure equations gothM} to see that
\[\begin{array}{l}
d\varphi^1 = -\frac{1}{4} (\varphi^1 \wedge \varphi^2 - \varphi^1 \wedge \bar{\varphi}^2 - \varphi^2 \wedge \bar{\varphi}^1 - \bar{\varphi}^1 \wedge \bar{\varphi}^2),\\
d\varphi^2 = -\frac{\ii}{4} (\varphi^1 \wedge \varphi^2 + \varphi^1 \wedge \bar{\varphi}^2 - \varphi^2 \wedge \bar{\varphi}^1 + \bar{\varphi}^1 \wedge \bar{\varphi}^2),
\end{array}\]
hence that $J$ is not integrable, and
\[\delbar(\varphi^1 \wedge \varphi^2) = -\frac{1}{4}(\ii \bar{\varphi}^1 + \bar{\varphi}^2) \wedge \varphi^1 \wedge \varphi^2.\]

\begin{prop}\label{prop: kod dim n 1}
We have the following:
\[\kod(\cN, J) = -\infty.\]
\end{prop}
\begin{proof}
By a direct computation, a smooth pluricanonical section $f (\varphi^1 \wedge \varphi^2)^{\tensor m}$ is pseudoholomorphic if and only if $\delbar f - \frac{1}{4} mf (\ii \bar{\varphi}^1 + \bar{\varphi}^2) = 0$, which is equivalent to the system
\[\left\{ \begin{array}{l}
\bar{\cX}_1(f) - \frac{1}{4} \ii mf = 0\\
\bar{\cX}_2(f) - \frac{1}{4} mf = 0.
\end{array}\right.\]
Writing $f = u + \ii v$ and using \eqref{eq: (1,0) fields} and \eqref{eq: tangent fields N} we find that the previous system is equivalent to
\[\left\{ \begin{array}{l}
\frac{\partial u}{\partial x} - \frac{\partial v}{\partial t} + \frac{1}{2}mv = 0\\
\frac{\partial v}{\partial x} + \frac{\partial u}{\partial t} - \frac{1}{2}mu = 0\\
\frac{\partial u}{\partial y} + x \frac{\partial u}{\partial z} + \frac{1}{2} x^2 \frac{\partial u}{\partial t} - \frac{\partial v}{\partial z} - x \frac{\partial v}{\partial t} - \frac{1}{2}mu = 0\\
\frac{\partial v}{\partial y} + x \frac{\partial v}{\partial z} + \frac{1}{2} x^2 \frac{\partial v}{\partial t} + \frac{\partial u}{\partial z} + x \frac{\partial u}{\partial t} - \frac{1}{2}mv = 0.
\end{array} \right.\]
Using the Fourier expansion \eqref{eq: Fourier N}, the last two equations in the last system become
\[\left\{ \begin{array}{l}
\ii \pi a u_I + \ii \pi b x u_I + \ii \pi c \frac{1}{2} x^2 u_I - \ii \pi b v_I - \ii \pi c x v_I - \frac{1}{2} m u_I = 0\\
\ii \pi a v_I + \ii \pi b x v_I + \ii \pi c \frac{1}{2} x^2 v_I + \ii \pi b u_I + \ii \pi c x u_I - \frac{1}{2} m v_I = 0,
\end{array} \right.\]
which is a homogeneous linear system for the pair $(u_I, v_I)$. The determinant of the matrix representing this last system is
\[-\pi^2 \pa{a + bx + \frac{1}{2}x^2 c + \frac{\ii}{2\pi}m + \ii (b + cx)} \pa{a + bx + \frac{1}{2}x^2 c + \frac{\ii}{2\pi}m - \ii (b + cx)}.\]
This determinant vanishes if and only if
\[\left\{ \begin{array}{l}
a + bx + \frac{1}{2}x^2 c = 0\\
\frac{m}{2\pi} + b + cx = 0
\end{array}\right. \qquad \text{or} \qquad \left\{ \begin{array}{l}
a + bx + \frac{1}{2}x^2 c = 0\\
\frac{m}{2\pi} - b - cx = 0,
\end{array}\right.\]
and we want to determine explicitly when this happens.
\begin{enumerate}
\item If $c = 0$, both systems are impossible since $\pi$ is irrational.
\item If $c \neq 0$, we can solve the second equation finding that $x = \frac{m}{2\pi c} - \frac{b}{c}$ (we focus on the second system, but the situation for the first one is completely analogous). We can then substitute this expression in the first equation, and after we clean the denominators we find the equation
\[4(2ca - b^2) \pi + m^2 = 0.\]
We have then two further subcases:
\begin{enumerate}
\item if $b^2 = 2ca$, this equation is never satisfied, since $m \neq 0$;
\item if $b^2 \neq 2ca$, this equation is never satisfied, since $\pi$ is irrational.
\end{enumerate}
\end{enumerate}
In any case, the determinant is non-zero for every value of $I = (a, b, c)$ and so $u_I$ and $v_I$ must always be identically zero. Obviously $u_I = v_I = 0$ is a solution of the system. As a consequence $u = v = 0$, which implies that the Kodaira dimension of $(\cN, J)$ is $-\infty$.
\end{proof}

\begin{rem}
Observe that with respect to the symplectic form $\xi$ we have that
\[\xi(J \blank, J \blank) = \xi(\blank, \blank),\]
and that $\xi = \frac{1}{2} \ii \pa{\varphi^1 \wedge \bar{\varphi}^1 + \varphi^2 \wedge \bar{\varphi}^2}$ is of pure type $(1, 1)$. Moreover, $\xi$ is the $(1, 1)$-form naturally associated with the metric
\[g' = e^1 \tensor e^1 + e^2 \tensor e^2 + e^3 \tensor e^3 + e^4 \tensor e^4\]
on $\cN$, and so $(\cN, J, g')$ is an almost K\"ahler manifold.
\end{rem}

We consider now a second almost complex structure on $\cN$: the one defined by
\[J' e_1 = e_2, \qquad J' e_2 = -e_1, \qquad J' e_3 = e_4, \qquad J' e_4 = -e_3.\]
In this case we have
\[\varphi^1 = e^1 + \ii e^2, \qquad \varphi^2 = e^3 + \ii e^4,\]
and so
\[\begin{array}{l}
d\varphi^1 = 0,\\
d\varphi^2 = -\frac{1}{4} \ii \pa{\varphi^1 \wedge \varphi^2 + 2 \varphi^1 \wedge \bar{\varphi}^1 + \varphi^1 \wedge \bar{\varphi}^2 - \varphi^2 \wedge \bar{\varphi}^1 + \bar{\varphi}^1 \wedge \bar{\varphi}^2}.
\end{array}\]
As a consequence
\[\delbar(\varphi^1 \wedge \varphi^2) = -\frac{1}{4} \ii \bar{\varphi}^1 \wedge \varphi^1 \wedge \varphi^2,\]
so for any smooth function $f: \cN \longrightarrow \IC$ we have
\begin{equation}\label{eq: pseudohol condition}
\begin{array}{rl}
\delbar \pa{f \cdot (\varphi^1 \wedge \varphi^2)^{\tensor m}} = & \delbar f \tensor (\varphi^1 \wedge \varphi^2)^{\tensor m} - \frac{1}{4} \ii mf \bar{\varphi}^1 \tensor (\varphi^1 \wedge \varphi^2)^{\tensor m} =\\
= & \pa{\delbar f - \frac{1}{4} \ii mf \bar{\varphi}^1} \tensor (\varphi^1 \wedge \varphi^2)^{\tensor m}.
\end{array}
\end{equation}

\begin{prop}\label{prop: kod dim n 2}
The Kodaira dimension of $(\cN, J')$ is
\[\kod(\cN, J') = -\infty.\]
\end{prop}
\begin{proof}
The dual frame of vector fields of type $(1, 0)$ on $(\cN, J')$ is
\[\cX_1 = \frac{1}{2} (e_1 - \ii e_2), \qquad \cX_2 = \frac{1}{2} (e_3 - \ii e_4),\]
so after we write $f = u + \ii v$ we see from \eqref{eq: pseudohol condition} that the smooth pluricanonical section $f \cdot (\varphi^1 \wedge \varphi^2)^{\tensor m}$ is pseudoholomorphic if and only if
\[\left\{ \begin{array}{l}
e_1(u) - e_2(v) + \frac{1}{2} mv = 0\\
e_1(v) + e_2(u) - \frac{1}{2} mu = 0\\
e_3(u) - e_4(v) = 0\\
e_3(v) + e_4(u) = 0.
\end{array} \right.\]
By \eqref{eq: tangent fields N} this system becomes
\begin{equation}\label{eq: system N}
\left\{ \begin{array}{l}
\frac{\partial u}{\partial x} - \frac{\partial v}{\partial y} - x \frac{\partial v}{\partial z} - \frac{1}{2} x^2 \frac{\partial v}{\partial t} + \frac{1}{2} mv = 0\\
\frac{\partial v}{\partial x} + \frac{\partial u}{\partial y} + x \frac{\partial u}{\partial z} + \frac{1}{2} x^2 \frac{\partial u}{\partial t} - \frac{1}{2} mu = 0\\
\frac{\partial u}{\partial z} + x \frac{\partial u}{\partial t} - \frac{\partial v}{\partial t} = 0\\
\frac{\partial v}{\partial z} + x \frac{\partial v}{\partial t} + \frac{\partial u}{\partial t} = 0.
\end{array} \right.
\end{equation}
Recall from \eqref{eq: Fourier N} that $u$ and $v$ can be expressed as Fourier series: in particular, the third and fourth equation of \eqref{eq: system N} give us the system for the Fourier coefficients $u_I$ and $v_I$ 
\[\left\{ \begin{array}{l}
(b + cx) u_I - c v_I = 0\\
(b + cx) v_I + c u_I = 0
\end{array} \right. \qquad \text{i.e.} \qquad \left( \begin{array}{cc}
b + cx & -c\\
c & b + cx
\end{array} \right) \left( \begin{array}{c}
u_I\\
v_I
\end{array} \right) = 0.\]
As the determinant of the above matrix is $(b + cx)^2 + c^2$, which is the sum of two \emph{real} squares, we argue that it vanishes if and only if $b = c = 0$. As a consequence
\[b \neq 0 \text{ or } c \neq 0 \quad\Longrightarrow\quad u_I = v_I = 0\]
and so $u$ and $v$ do not depend on $z$ and $t$. It is then easy to check that they admit a `new' Fourier series expansion as
\[u(x, y) = \sum_{(\lambda, \mu) \in \IZ^2} u_{\lambda\mu} e^{\ii \pi (\lambda x + \mu y)}, \qquad u_{\lambda\mu} \in \IC\]
and similarly for $v$. Moreover, the first and second equations of \eqref{eq: system N} simplify to
\[\left\{ \begin{array}{l}
\frac{\partial u}{\partial x} = \frac{\partial v}{\partial y} - \frac{1}{2} mv\\
\frac{\partial u}{\partial y} = - \frac{\partial v}{\partial x} + \frac{1}{2} mu
\end{array} \right.\]
and so
\[\frac{\partial^2 v}{\partial y^2} - \frac{1}{2}m \frac{\partial v}{\partial y} = \frac{\partial u}{\partial y \partial x} = \frac{\partial u}{\partial x \partial y} = - \frac{\partial^2 v}{\partial x^2} + \frac{1}{2} m \pa{\frac{\partial v}{\partial y} - \frac{1}{2} mv}.\]
So, both $u$ and $v$ are solution of the second order operator
\[\frac{\partial^2}{\partial x^2} + \frac{\partial^2}{\partial y^2} - m \frac{\partial}{\partial y} + \frac{1}{4}m^2,\]
which implies that the Fourier coefficients satisfy
\[\pa{-\pi^2 \lambda^2 - \pi^2 \mu^2 + \frac{1}{4}m^2 - \ii \pi m \mu} u_{\lambda\mu} = 0.\]
Observe that if $\mu \neq 0$ then the coefficient in the above equation has non-vanishing imaginary part, which forces $u_{\lambda\mu} = 0$. Assume finally that $\mu = 0$: the previous relation simplifies to $\pa{-\pi^2 \lambda^2 + \frac{1}{4}m^2} u_{\lambda 0} = 0$, which is still satisfied only if $u_{\lambda 0} = 0$ by the irrationality of $\pi$. This means that $u = 0$ and analogously for $v$.

To sum up, we have shown that if $f = u + \ii v$ is such that $f \cdot (\varphi^1 \wedge \varphi^2)^{\tensor m}$ is a pseudoholomorphic pluricanonical section then $f = 0$, which means that
\[P_m(\cN, J') = 0 \qquad \forall m \geq 1\]
and so
\[\kod(\cN, J') = -\infty.\]
\end{proof}

\begin{rem}
With respect to this second almost complex structure we have
\[\xi = -\frac{1}{2} \ii \pa{\varphi^1 \wedge \varphi^2 - \bar{\varphi}^1 \wedge \bar{\varphi}^2}\]
and so $\xi(J' \blank, J' \blank) = -\xi(\blank, \blank)$.
\end{rem}

Finally, consider the following family of almost complex structures on $\cN$:
\[J_ae_1 = \frac{1}{a} e_3, \qquad J_a e_2 = e_4, \qquad J_a e_3 = -ae_1, \qquad J_a e_4 = -e_2.\]
where $a\in \mathbb R \setminus \{0\}.$\\
In this case we have the $(1, 0)$-fields
\[\cX_{a1} = \frac{1}{2} \pa{e_1 - \frac{\ii}{a} e_3}, \qquad \cX_{a2} = \frac{1}{2} (e_2 - \ii e_4)\]
with dual $(1, 0)$-forms
\[\varphi_a^1 = e^1 + \ii ae^3, \qquad \varphi_a^2 = e^2 + \ii e^4.\]
In this case, the canonical section $\varphi_a^1 \wedge \varphi_a^2$ is not pseudoholomorphic as
\[\delbar({\varphi_a^1} \wedge \varphi_a^2) = \frac{a}{4} \ii {\varphi_a^1} \wedge \varphi_a^2 \wedge \bar{\varphi}_a^2.\]

We compute the Kodaira dimension of $(\cN, J_a)$ in the next proposition.

\begin{prop}\label{prop: kod dim n 3}
We have
\[\kod(\cN, J_a) = \left\{ \begin{array}{ll}
-\infty & \text{if } a \notin 2 \pi \IQ,\\
0 & \text{if } a \in 2 \pi \IQ \smallsetminus \{ 0 \}.
\end{array} \right.\]
\end{prop}
\begin{proof}
It is easy to see that a pluricanonical form $f (\varphi_a^1 \wedge \varphi_a^2)^{\tensor k}$ is pseudoholomorphic if and only if
\[\left\{ \begin{array}{l}
\bar{\cX}_{a1}(f) = 0\\
\bar{\cX}_{a2}(f) + \frac{1}{4} \ii a k f = 0.
\end{array} \right.\]
Writing $f = u + \ii v$ and using \eqref{eq: tangent fields N}, the previous systems is equivalent to
\[\left\{ \begin{array}{l}
\frac{\partial u}{\partial x} - \frac{1}{a}(\frac{\partial v}{\partial z} + x \frac{\partial v}{\partial t}) = 0\\
\frac{\partial v}{\partial x} + \frac{1}{a}(\frac{\partial u}{\partial z} + x \frac{\partial u}{\partial t}) = 0\\
2 \frac{\partial u}{\partial y} + 2x \frac{\partial u}{\partial z} + x^2 \frac{\partial u}{\partial t} - 2 \frac{\partial v}{\partial t} - kav = 0\\
2 \frac{\partial v}{\partial y} + 2x \frac{\partial v}{\partial z} + x^2 \frac{\partial v}{\partial t} + 2 \frac{\partial u}{\partial t} + kau = 0.
\end{array} \right.\]

By \eqref{eq: Fourier N} the last differential system becomes then a differential system for the Fourier coefficients $u_N(x)$ and $v_N(x)$: denoting by $'$ the derivative with respect to $x$ we have
\[\left\{ \begin{array}{l}
u_N' - \frac{1}{a}\ii \pi m v_N - \frac{1}{a}\ii \pi p x v_N = 0\\
v_N' + \frac{1}{a}\ii \pi m u_N + \frac{1}{a}\ii \pi p x u_N = 0\\
2\pi \ii n u_N + 2\pi \ii mx u_N + \pi \ii p x^2 u_N - 2\pi \ii p v_N - ka v_N = 0\\
2\pi \ii n v_N + 2\pi \ii mx v_N + \pi \ii p x^2 v_N + 2\pi \ii p u_N + ka u_N = 0.
\end{array} \right.\]
The last two equations involve only $u_N$ and $v_N$
\[\left( \begin{array}{cc}
\ii \pi (2n + 2mx + px^2) & -(ka + 2\pi \ii p)\\
ka + 2\pi \ii p & \ii \pi (2n + 2mx + px^2)
\end{array} \right) \left( \begin{array}{c}
u_N\\
v_N
\end{array} \right) = \left( \begin{array}{c}
0\\
0
\end{array} \right),\]
and so we can deduce what follows.
\begin{enumerate}
\item If $p \neq 0$, then $u_N = v_N = 0$. In fact, the determinant of the previous matrix is not zero, as its imaginary part is $4\pi p \neq 0$.
\item If $p = 0$ and $m \neq 0$, then $u_N = v_N = 0$. In this case we see that the last system becomes
\[\left\{ \begin{array}{l}
v_N = \frac{2\pi}{ka} \ii (n + mx) u_N\\
(k^2a^2 - 4\pi^2(n + mx)^2) u_N = 0.
\end{array} \right.\]
From the second equation we deduce that $u_N$ must be zero, except possibly for $x = -\frac{2\pi n \pm ka}{2\pi m}$ where the first function vanishes. As $u_N$ is continuous, we deduce that $u_N = v_N = 0$.
\item If $m = p = 0$, then the determinant of the previous matrix equals $k^2a^2 - 4\pi^2 n^2$. This is non zero for $a \notin 2\pi \mathbb Q \setminus \{0\}$, so in this case we have that $u_N = v_N = 0$.
\end{enumerate}

This shows that if $a \notin 2\pi \IQ$ then $u_N = v_N = 0$ for every $N \in \IZ^3$ and so $\kod(\cN, J_a) = 0$.

Assume now that $a \in 2\pi\IQ \smallsetminus \{ 0 \}$, hence that $a = 2\pi \frac{n'}{k'}$ for some $n', k' \in \IZ$ with $ k' > 0$ and $\gcd(n', k') = 1$. What we did in the previous section allow us to conclude the following:
\begin{enumerate}
\item if $k$ is not a multiple of $k'$, then $u_N = v_N = 0$ for every $N \in \IZ^3$ and so $P_k(\cN, J_a) = 0$;
\item if $k$ is a multiple of $k'$, say $k = qk'$ for a suitable positive integer $q$, then $u_N = v_N = 0$ for every $N \in \IZ^3$, possibly with the exception of $N = (n, m, p) = (\pm qn', 0, 0)$.
\end{enumerate}
Hence we focus on this last case. For $k = q k'$ and $N = (n, m, p) = (qn', 0, 0)$ we have the system
\[\left\{ \begin{array}{l}
u'_N = 0\\
v'_N = 0\\
\ii u_N - v_N = 0\\
\ii v_N + u_N = 0
\end{array} \right. \longrightarrow \left\{ \begin{array}{l}
u_N(x) \equiv u_N \in \IC\\
v_N(x) \equiv v_N \in \IC\\
v_N = \ii u_N.
\end{array} \right.\]
The analogous system for $N = (-qn', 0, 0)$ has $v_N = -\ii u_N$ as solution, where $u_N(x) \equiv u_N$ and $v_N(x) \equiv v_N$ are again constant.

Since $u$ and $v$ are real valued functions, we have that $u_{-N} = \bar{u}_N$ and $v_{-N} = \bar{v}_N$, so
\[\begin{array}{rl}
u(x, y, z, t) = & u_{(-qn', 0, 0)} e^{-\ii \pi q n' y} + u_{(qn', 0, 0)} e^{\ii \pi q n' y}\\
= & \bar{u}_{(qn', 0, 0)} e^{-\ii \pi q n' y} + u_{(qn', 0, 0)} e^{\ii \pi q n' y}\\
= & 2 \Re \pa{u_{(qn', 0, 0)} e^{\ii \pi q n' y}},\\
v(x, y, z, t) = & -\ii u_{(-qn', 0, 0)} e^{-\ii \pi q n' y} + \ii u_{(qn', 0, 0)} e^{\ii \pi q n' y}\\
= & \ii \pa{u_{(qn', 0, 0)} e^{\ii \pi q n' y} - \bar{u}_{(qn', 0, 0)} e^{-\ii \pi q n' y}}\\
= & -2 \Im \pa{u_{(qn', 0, 0)} e^{\ii \pi q n' y}}.
\end{array}\]
As a consequence we can deduce that
\[f(x, y, z, t) = 2 \bar{u}_{(qn', 0, 0)} e^{-\ii \pi q n' y},\]
and it is now easy to verify that such a function satisfies the periodicity prescribed by \eqref{eq: periodicity N} and so descends from $\IR^4$ to $\cN$.

To sum up, we have shown that for the almost complex manifold $X_a = (\cN, J_a)$, if $a = 2 \pi \frac{n'}{k'}$ for coprime integers $n', k'$ with $k' > 0$ then for every positive $q \in \IZ$
\[H^0(X_a, \omega_{X_a}^{\tensor qk'}) = \IC \cdot e^{-\ii \pi q n' y}.\]
Hence $P_{qk'}(\cN, J_a) = 1$ and so
\[\kod(\cN, J_a) = 0.\]
\end{proof}

\subsection{Curvature of the canonical connection}
Let $J$ be the almost complex structure on $\cN$ defined by \eqref{almost-kaehler-nil}. Then the action of $J$ on the the dual coframe $\{e^1,\ldots,e^4\}$ of $\{e_1,\ldots,e_4\}$ is given by 
$$
Je^1=-e^4,\qquad Je^2=-e^3,\qquad Je^3=e^2,\qquad Je^4=e^1.
$$
Then, setting
$$
\omega=e^{1}\wedge e^4+e^2\wedge e^3,\qquad g_J(\cdot,\cdot)=\omega(\cdot,J\cdot),
$$
the pair $(J,g_J)$ gives rise to an almost Hermitian metric on $\cN$. A straightforward computation yields to the following

\begin{prop}\label{scalar-flat}
On the nilmanifold $\cN$ consider the almost Hermitian metric $(J, g_J)$. Then $(J, g_J)$ is a non Chern--Ricci flat almost K\"ahler metric with vanishing scalar curvature on $\cN$.
\end{prop}

\begin{proof}
The coframe given by $\Phi^1 = \frac{\sqrt{2}}{2} \varphi^1$ and $\Phi^2 = \frac{\sqrt{2}}{2} \varphi^2$ is unitary with respect to the Hermitian metric induced by $(J, g_J)$.

One can then compute the Nijenhuis tensor of $J$ and use it to compute the real torsion forms of the canonical connection (see \cite[Corollary 3.8]{Cattaneo-Nannicini-Tomassini}). Hence the complex torsion forms are
\[\Theta^1 = \frac{\sqrt{2}}{4} \bar{\Phi}^1 \wedge \bar{\Phi}^2, \qquad \text{and} \qquad \Theta^2 = -\frac{\sqrt{2}}{4} \ii \bar{\Phi}^1 \wedge \bar{\Phi}^2.\]
Solving the primary structure equations to find the connection forms $\vartheta^i_j$ we see that these are given by
\[\begin{array}{ll}
\vartheta^1_1 = -\frac{\sqrt{2}}{4} \Phi^2 + \frac{\sqrt{2}}{4} \bar{\Phi}^2, & \vartheta^1_2 = -\frac{\sqrt{2}}{4} \ii \Phi^2 + \frac{\sqrt{2}}{4} \bar{\Phi}^1,\\
\vartheta^2_1 = -\frac{\sqrt{2}}{4} \Phi^1 - \frac{\sqrt{2}}{4} \ii \bar{\Phi}^2, & \vartheta^2_2 = \frac{\sqrt{2}}{4} \ii \Phi^1 + \frac{\sqrt{2}}{4} \bar{\Phi}^1.
\end{array}\]

The secondary structure equations allow us to compute the curvature forms $\psi^i_j$:
\[\begin{array}{l}
\psi^1_1 = \frac{1}{8} (\ii \Phi^1 \wedge \Phi^2 + \Phi^1 \wedge \bar{\Phi}^1 + 2\ii \Phi^1 \wedge \bar{\Phi}^2 - 2\ii \Phi^2 \wedge \bar{\Phi}^1 - \Phi^2 \wedge \bar{\Phi}^2 + \ii \bar{\Phi}^1 \wedge \bar{\Phi}^2),\\
\psi^1_2 = -\frac{1}{8} (\Phi^1 \wedge \Phi^2 + \ii \Phi^1 \wedge \bar{\Phi}^1 + 2 \Phi^1 \wedge \bar{\Phi}^2 - \ii \Phi^2 \wedge \bar{\Phi}^2 + 3 \bar{\Phi}^1 \wedge \bar{\Phi}^2),\\
\psi^2_1 = \frac{1}{8}(3 \Phi^1 \wedge \Phi^2 + \ii \Phi^1 \wedge \bar{\Phi}^1 - 2 \Phi^2 \wedge \bar{\Phi}^1 - \ii \Phi^2 \wedge \bar{\Phi}^2 + \bar{\Phi}^1 \wedge \bar{\Phi}^2),\\
\psi^2_2 = \frac{1}{8}(\ii \Phi^1 \wedge \Phi^2 - \Phi^1 \wedge \bar{\Phi}^1 + \Phi^2 \wedge \bar{\Phi}^2 + \ii \bar{\Phi}^1 \wedge \bar{\Phi}^2).
\end{array}\]
and finally the Chern--Ricci tensor of the canonical connection, which is:
\[\frac{1}{4} \ii \Phi^1 \wedge \bar{\Phi}^2 - \frac{1}{4} \ii \Phi^2 \wedge \bar{\Phi}^1.\]
\end{proof}

\section{A special almost complex structure on \texorpdfstring{$\cM(\lambda)$}{M(l)}}\label{s7}

Let $J$ be the almost complex structure on $\cM(\lambda)$ defined on the tangent fields \eqref{eq: tangent vectors cm(lambda)} by
\[J e_1 = e_2, \qquad J e_2 = -e_1, \qquad J e_3 = e_4, \qquad J e_4 = -e_3,\]
and let $X = (\cM(\lambda), J)$ be the corresponding almost complex manifold. From the complex point of view, we introduce the $(0, 1)$-forms
\[\varphi^1 = e^1 + \ii e^2, \qquad \varphi^2 = e^3 + \ii e^4,\]
and thanks to the structure equations \eqref{eq: structure equations cm(lambda)} we see that
\[\begin{array}{l}
d\varphi^1 = \frac{1}{2} \varphi^1 \wedge \bar{\varphi}^1,\\
d\varphi^2 = \frac{1}{4} \varphi^1 \wedge \varphi^2 - \frac{1}{4}(1 + 2\lambda) \varphi^1 \wedge \bar{\varphi}^2 - \frac{1}{4} \varphi^2 \wedge \bar{\varphi}^1 - \frac{1}{4}(1 + 2\lambda) \bar{\varphi}^1 \wedge \bar{\varphi}^2.
\end{array}\]

Hence
\[\delbar(\varphi^1 \wedge \varphi^2) = -\frac{1}{4} \varphi^1 \wedge \varphi^2 \wedge \bar{\varphi}^1,\]
and so for a smooth complex valued function $f: \cM(\lambda) \longrightarrow \IC$ we have that
\begin{equation}\label{eq: pseudohol cond M(lambda)}
\delbar\pa{f\cdot(\varphi^1 \wedge \varphi^2)^{\tensor m}} = \pa{\delbar f - \frac{1}{4}mf \bar{\varphi}^1} \tensor (\varphi^1 \wedge \varphi^2)^{\tensor m}.
\end{equation}

\begin{prop}\label{prop: kod dim m(lambda)}
The Kodaira dimension of $(\cM(\lambda), J)$ is
\[\kod(\cM(\lambda), J) = -\infty.\]
\end{prop}
\begin{proof}
It follows from \eqref{eq: pseudohol cond M(lambda)} that a smooth pluricanonical section of the form $f\cdot(\varphi^1 \wedge \varphi^2)^{\tensor m}$, with $m \geq 1$, is pseudoholomorphic if and only if $\delbar f - \frac{1}{4}mf \bar{\varphi}^1 = 0$. As the vector fields of type $(1, 0)$ corresponding to $\varphi^1$, $\varphi^2$ are
\[\cX_1 = \frac{1}{2}\pa{e_1 - \ii e_2}, \qquad \cX_2 = \frac{1}{2}\pa{e_3 - \ii e_4},\]
the previous equation is equivalent to the system
\[\left\{ \begin{array}{l}
\bar{\cX}_1(f) - \frac{1}{4}mf = 0\\
\bar{\cX}_2(f) = 0.
\end{array} \right.\]
Using the definition \eqref{eq: tangent vectors cm(lambda)} for the fields $e_i$ and writing $f = u + \ii v$, once we take the real and the imaginary parts of the equations in the previous system we find that it becomes equivalent to
\begin{equation}\label{eq: system m(lambda)}
\left\{ \begin{array}{l}
\frac{\partial u}{\partial t} - e^t \frac{\partial v}{\partial y} - \frac{1}{2} mu = 0\\
\frac{\partial v}{\partial t} + e^t \frac{\partial u}{\partial y} - \frac{1}{2} mv = 0\\
e^{\lambda t} \frac{\partial u}{\partial x} - e^{-(1 + \lambda)t} \frac{\partial v}{\partial z} = 0\\
e^{\lambda t} \frac{\partial v}{\partial x} + e^{-(1 + \lambda)t} \frac{\partial u}{\partial z} = 0.
\end{array} \right.
\end{equation}
From the third and fourth equations of this system we see that
\[e^{(1 + 2\lambda)t} \frac{\partial^2 v}{\partial x^2} + e^{-(1 + 2\lambda)t} \frac{\partial^2 v}{\partial z^2} = 0\]
and similarly for $u$. As a consequence, both $u$ and $v$ do not depend neither on $x$ nor on $z$.

Thanks to this information, the Fourier series expansion \eqref{eq: Fourier M(lambda)} of $u$ and $v$ simplify to
\[u = \sum_{k \in \IZ} u_k(t) e^{2\pi \ii k y}, \qquad v = \sum_{k \in \IZ} v_k(t) e^{2\pi \ii k y}.\]
In fact, if $u_I(t) \neq 0$ (resp.~$v_I(t) \neq 0$) then we must have $\lambda_I = \nu_I = 0$ because otherwise $u$ (resp.~$v$) would depend on $x$ or $z$. Hence the sum is taken over $I \in \IZ^3$ such that $(\lambda_I, \mu_I, \nu_I) = (0, k, 0)$, i.e., over all the integral multiples of the second row of the matrix $P$ associated to the lattice.
The first two equations of \eqref{eq: system m(lambda)} then give us the differential system for the Fourier coefficients $u_k$, $v_k$
\[\left( \begin{array}{c}
\frac{\partial u_k}{\partial t}\\
\frac{\partial v_k}{\partial t}
\end{array} \right) = \left( \begin{array}{cc}
\frac{1}{2}m & 2\pi\ii e^t k\\
-2\pi\ii e^t k & \frac{1}{2}m
\end{array} \right) \left( \begin{array}{c}
u_k\\
v_k
\end{array} \right).\]
If we make the substitution
\[\left( \begin{array}{c}
\xi_k\\
\zeta_k
\end{array} \right) = \left( \begin{array}{cc}
1 & \ii\\
-1 & \ii
\end{array} \right) \left( \begin{array}{c}
u_k\\
v_k
\end{array} \right),\]
then the system decouples and we obtain
\[\left\{ \begin{array}{l}
\frac{\partial \xi_k}{\partial t} = \pa{\frac{1}{2}m + 2\pi k e^t} \xi_k\\
\frac{\partial \zeta_k}{\partial t} = \pa{\frac{1}{2}m - 2\pi k e^t} \zeta_k,
\end{array} \right.\]
which can easily be solved. Explicitly, the solution $(u_k, v_k)$ is
\[\begin{array}{l}
u_k(t) = \frac{1}{2} e^{\frac{1}{2}mt} \pa{c_k e^{2\pi k e^t} - d_k e^{-2\pi k e^t}},\\
v_k(t) = -\frac{\ii}{2} e^{\frac{1}{2}mt} \pa{c_k e^{2\pi k e^t} + d_k e^{-2\pi k e^t}},
\end{array}\]
with $c_k, d_k \in \IC$. We can then see that
\[f = f(y, t) = \sum_{k \in \IZ} f_k(t) e^{2\pi \ii k y}, \qquad f_k(t) = u_k(t) + \ii v_k(t) = e^{\frac{1}{2}mt} c_k e^{2\pi k e^t}.\]
Recall from \eqref{eq: periodicity M(lambda)} that in order for $f$ to define a function on $\cM(\lambda)$ we must have that $f(e^\delta y, t + \delta) = f(y, t)$ for every $\delta \in \IZ \cdot t_0$. For any fixed $\delta \in \IZ \cdot t_0$ we can write
\[f(e^\delta y, t + \delta) = \sum_{h \in \IZ} g_h^\delta(t) e^{2 \pi \ii h y},\]
where the Fourier coefficients $g_h^\delta(t)$ can be explicitly computed:
\[g_h^\delta(t) = \int_0^1 f(e^\delta y, t + \delta) e^{-2\pi \ii h y} dy = e^{\frac{1}{2}m(t + \delta)} \sum_{k \in \IZ} c_k e^{2\pi k e^{t + \delta}} I(k e^\delta - h),\]
where we put
\[I(\eta) = \int_0^1 e^{2\pi \ii \eta y} dy = \left\{ \begin{array}{ll}
1 & \text{for } \eta = 0,\\
\frac{1}{2\pi \ii \eta} \pa{e^{2\pi \ii \eta} - 1} & \text{for } \eta \neq 0.
\end{array} \right.\]
Because of the invariance with respect to the lattice we have $f_h(t) = g_h^\delta(t)$ for every $\delta \in \IZ \cdot t_0$, so in particular $f_h = \lim_{\delta \rightarrow -\infty} g_h^\delta$. Since
\[\begin{array}{l}
e^{\frac{1}{2} m (t + \delta)} \longrightarrow 0\\
e^{ 2\pi k e^{t + \delta}} \longrightarrow 1\\
I(k e^\delta - h) \longrightarrow I(-h) = \left\{ \begin{array}{ll}
1 & \text{for } h = 0,\\
0 & \text{for } h \neq 0
\end{array} \right.
\end{array}\]
we deduce that
\[f_h(t) = \lim_{\delta \rightarrow -\infty} g_h^\delta(t) = 0.\]
It follows that then $f \equiv 0$, the plurigenera of $(\cM(\lambda), J)$ are
\[P_m(\cM(\lambda)) = 0 \qquad \forall m\geq 1,\]
hence
\[\kod(\cM(\lambda), J) = -\infty.\]
\end{proof}

\section{A twistorial approach}\label{sect: twistor}
In the following we will denote by $M$ one of the three four dimensional solvmanifolds without complex structures described before, unless otherwise specified. Moreover we will denote by $\{e_1,e_2,e_3,e_4\}$ and $\{e^1,e^2,e^3,e^4\}$ the global frame for $TM$ and $T^*M$ respectively.

\subsection{Almost hypercomplex structure on \texorpdfstring{$M(k)$}{M(k)}, \texorpdfstring{$\cN$}{M}, \texorpdfstring{$\cM(\lambda)$}{M(l)}}

Consider the following  endomorphisms of the tangent bundle of $M$:
\[J_0 = \left( \begin{array}{cccc}
0 &-1 & 0 & 0\\
1 & 0 & 0 & 0\\
0 & 0 & 0 & -1\\
0 & 0 & 1 & 0
\end{array} \right), \, J_1 = \left( \begin{array}{cccc}
0 &0 & -1 & 0\\
0 & 0 & 0 & 1\\
1 & 0 & 0 & 0\\
0 & -1 & 0 & 0
\end{array} \right). \] 
It is  easy to see that $J_0 J_1 = -J_1 J_0$, thus we define:
\[J_2 = J_0 J_1 = \left( \begin{array}{cccc}
0 & 0 & 0 & -1\\
0 & 0 & -1 & 0\\
0 & 1 & 0 & 0\\
1 & 0 & 0 & 0
\end{array} \right) \] 
and we get that $J_0, J_1, J_2$ define an almost hypercomplex structure on $M$.

Let $\omega_0$ and  $g_0$ be defined by:
$$\omega_0 = e^1 \wedge e^2 + e^3 \wedge e^4$$
$$g_0= e^1 \tensor e^1 + e^2 \tensor e^2 + e^3 \tensor e^3 + e^4 \tensor e^4.$$
We have
\[g_0(\blank, \blank) = \omega_0(\blank, J_0 \blank).\]
Then we consider $$\omega_1(\blank, \blank) = g_0(J_1 \blank, \blank) = e^1 \wedge e^3 - e^2 \wedge e^4$$ and $$\omega_2(\blank, \blank) = g_0(J_2 \blank, \blank) =e^1 \wedge e^4 +e^2 \wedge e^3.$$

A direct computation gives the following results.

\begin{lemma}
$g_0(J_i \blank, J_i\blank) = g_0(\blank, \blank)$ for $i = 0, 1, 2$.
\end{lemma}

\begin{lemma}
$\omega_0, \omega_1, \omega_2$ are self-dual. 
\end{lemma}

\begin{lemma}
If $M$=\texorpdfstring{$M(k)$}{M(k)} then $d\omega_0 = 0, d\omega_1 = k e^2 \wedge e^3 \wedge e^4, d\omega_2 = -k e^1 \wedge e^3 \wedge e^4$.\\
If $M$= \texorpdfstring{$\cN$}{M} then $d\omega_0 =-e^1\wedge e^2\wedge e^4, \, d\omega_1 =e^1 \wedge e^2 \wedge e^3, \,d\omega_2 =0$.\\
If $M$=\texorpdfstring{$\cM(\lambda)$}{M(l)} then $d\omega_0 =e^1\wedge e^3\wedge e^4, \, d\omega_1 =-\lambda e^1 \wedge e^2 \wedge e^4, \,d\omega_2 =-(1+\lambda)e^1 \wedge e^2 \wedge e^3$.
\end{lemma}

\subsection{The twistor bundle of \texorpdfstring{$M(k)$}{M(k)}, \texorpdfstring{$\cN$}{N}, \texorpdfstring{$\cM(\lambda)$}{M(l)}}

Let us consider the Riemannian manifold $(M, g_0)$ with the orientation defined by $e_1, e_2, e_3, e_4$. Let $P_{g_0} = P_{g_0}(M,SO(4))$ be the $SO(4)$-principal bundle of oriented $g_0$-orthonormal frames on $M$. $SO(4)$ acts on the right on $P_{g_0}$ and on the left on $SO(4)/U(2)$. The \emph{twistor space} of $(M, g_0)$ is the associated bundle to $P_{g_0}$ defined as the quotient $Z_{g_0} = Z(M, g_0)$ of $P_{g_0} \times SO(4)/U(2)$ with respect to previous action. $Z_{g_0}$ is a trivial bundle over $M$ with fibre $SO(4)/U(2)$. Let $x \in M$, the fibre $(Z_{g_0})_x$ parametrises the complex structures on $T_x M$ compatible with the metric $g_0$ and the fixed orientation. A global section is an almost complex structure on $M$ compatible with the metric and the orientation. For more details on the construction and main properties of the twistor bundle, we address the interested reader to \cite{deBartolomeis-Nannicini}.

By using the twistorial description we have immediately that any almost complex structure $J$ on $M$, compatible with the metric $g_0$ and the fixed orientation, is given by:
\[J = \alpha J_0 + \beta J_1 + \gamma J_2\]
where $\alpha, \beta, \gamma$ are smooth functions on $M$ such that $\alpha^2 + \beta^2 + \gamma^2 = 1$.

\begin{prop}
Let $J$ be an almost complex structure on $M$=\texorpdfstring{$M(k)$}{M(k)} compatible with the metric $g_0$ and the given orientation and let $\omega$ be the K\"ahler form of $J$, $\omega(\blank, \blank):= g_0(J \blank, \blank)$, then $d\omega = 0$ if and only if $\omega = c\omega_0$ for some real constant $c$.
\end{prop}
\begin{proof}
Let $J = \alpha J_0 + \beta J_1 + \gamma J_2$, then $\omega = \alpha \omega_0 + \beta \omega_1 + \gamma \omega_2$. In particular $d\omega = 0$ implies $d(\star \omega)=0$, where $\star$ is the Hodge operator defined by $g_0$. Hence $\omega$ is harmonic and then $\omega = c_1 e^1 \wedge e^2 + c_2 e^3 \wedge e^4$, for some real constants $c_1, c_2$. On the other hand $\omega = \alpha e^1 \wedge e^2 + \alpha e^3 \wedge e^4$, then $\alpha = c_1 = c_2 = c$ is constant and $\omega = c \omega_0$.
\end{proof}
Analogously
\begin{prop}
Let $J$ be an almost complex structure on $M$=\texorpdfstring{$\cN$}{M} compatible with the metric $g_0$ and the given orientation and let $\omega$ be the K\"ahler form of $J$, $\omega(\blank, \blank):= g_0(J \blank, \blank)$, then $d\omega = 0$ if and only if $\omega = c\omega_2$ for some real constant $c$.
\end{prop}
\begin{proof}
Let $J = \alpha J_0 + \beta J_1 + \gamma J_2$, then $\omega = \alpha \omega_0 + \beta \omega_1 + \gamma \omega_2$. In particular $d\omega = 0$ implies $d(\star \omega)=0$, where $\star$ is the Hodge operator defined by $g_0$. Hence $\omega$ is harmonic and then $\omega = c_1 e^1 \wedge e^4 + c_2 e^2 \wedge e^3$, for some real constants $c_1, c_2$. On the other hand $\omega =-\gamma e^1 \wedge e^4 - \gamma e^2 \wedge e^3$, then $\gamma =-c_1 =-c_2 = c$ is constant and $\omega = c \omega_2$.
\end{proof}

\begin{rem}
For $M$=\texorpdfstring{$\cM(\lambda)$}{M(l)} it is impossible to have $J$ such that $d\omega=0$, so there is not an analogous proposition in this case.
\end{rem}
\subsection{Kodaira dimension of a special family}

Define
\[J_{(a, b, c)} = a J_0 + b J_1 + c J_2, \qquad a, b, c \in \IR \text{ such that } a^2 + b^2 + c^2 = 1.\]
Then $J_{(a, b, c)}$ is a (constant) section of the twistor bundle $Z_{g_0}$ and so defines a $g_0$-compatible almost complex structure. Among the members of this family we can find the almost complex structures $J_0$ (for $(a, b, c) = (1, 0, 0)$), $J_1$ (for $(a, b, c) = (0, 1, 0)$) and $J_2$ (for $(a, b, c) = (0, 0, 1)$), and the only almost K\"ahler ones are $J_0$, $-J_0$ for $M$=\texorpdfstring{$M(k)$}{M(k)} and $J_2$, $-J_2$ for $M$=\texorpdfstring{$\cN$}{M} . As we know that $\operatorname{kod}(M(k), \pm J_0) = 0$, $\operatorname{kod}(\cN, \pm J_0) = -\infty$ and $\operatorname{kod}(\cM(\lambda), \pm J_0) = -\infty$ we can (and will) assume that $a \neq \pm 1$.

We observe that a $J$-adapted basis of vector fields is given by
\[\begin{array}{ll}
\varepsilon_1 = b e_1 + c e_2 - a e_4, & \varepsilon_3 = c e_1 - b e_2 + a e_3,\\
\varepsilon_2 = e_3, & \varepsilon_4 = e_4,
\end{array}\]
hence dually we have
\[\begin{array}{ll}
\varepsilon^1 = \frac{b}{1 - a^2} e^1 + \frac{c}{1 - a^2} e^2, & \varepsilon^3 = \frac{c}{1 - a^2} e^1 - \frac{b}{1 - a^2} e^2\\
\varepsilon^2 = -\frac{ac}{1 - a^2} e^1 + \frac{ab}{1 - a^2} e^2 + e^3, & \varepsilon^4 = \frac{ab}{1 - a^2} e^1 + \frac{ac}{1 - a^2} e^2 + e^4.
\end{array}\]

We can then consider the vector fields and $1$-forms of type $(1, 0)$ given respectively by
\[\cX_1 = \frac{1}{2}(\varepsilon_1 - \ii \varepsilon_2), \qquad \cX_2 = \frac{1}{2}(\varepsilon_3 - \ii \varepsilon_4)\]
and
\[\varphi^1 = \varepsilon^1 + \ii \varepsilon^2, \qquad \varphi^2 = \varepsilon^3 + \ii \varepsilon^4.\]

In the following subsections we will compute the Kodaira dimension of $(M(k), J_{(a, b, 0)})$, $(\cN, J_{(a, b, 0)})$ and $(\cM(\lambda), J_{(a, b, 0)})$. The computations for the general case (with $J_{(a, b, c)}$) are more complicated but, in principle, they can be treated with the same methods we have presented so far. We plan to come back on this subject in a future work.

\subsubsection{Kodaira dimension of $(M(k), J_{(a, b, c)})$} First of all we compute the Kodaira dimension of $(M(k), J_{(a, b, c)})$.

We have that
\[\begin{array}{rl}
\delbar \varphi^1 = & \frac{k}{2(1 - a^2)} ((2 acb + \ii (b^2 - c^2)) \varphi^1 \wedge \bar{\varphi}^1 + \ii bc (1 - a^2) \varphi^1 \wedge \bar{\varphi}^2 +\\
 & - (a(b^2 - c^2) - \ii bc (1 + a^2)) \varphi^2 \wedge \bar{\varphi}^1)
\end{array}\]
and
\[\begin{array}{rl}
\delbar \varphi^2 = & \frac{k}{4(1 - a^2)} (2(-a(b^2 - c^2) + 2 \ii bc) \varphi^1 \wedge \bar{\varphi}^1 + \ii (b^2 - c^2)(a^2 - 1) \varphi^1 \wedge \bar{\varphi}^2 +\\
 & - (4 abc + \ii (b^2 - c^2)(a^2 + 1)) \varphi^2 \wedge \bar{\varphi}^1),
\end{array}\]
from which
\[\delbar(\varphi^1 \wedge \varphi^2) = -\frac{k}{4}\ii \pa{(b^2 - c^2) \varphi^1 \wedge \varphi^2 \wedge \bar{\varphi}^1 + 2 bc \varphi^1 \wedge \varphi^2 \wedge \bar{\varphi}^2}.\]
As a section of $\Omega^{0, 1}_X \tensor \Omega^{2, 0}_X$ this corresponds to $\delbar(\varphi^1 \wedge \varphi^2) = \omega \tensor (\varphi^1 \wedge \varphi^2)$, where
\[\omega = -\frac{k}{4} \ii \pa{(b^2 - c^2) \bar{\varphi}^1 + 2 bc \bar{\varphi}^2}\]
and so
\[\delbar \pa{(\varphi^1 \wedge \varphi^2)^{\tensor m}} = m\omega \tensor (\varphi^1 \wedge \varphi^2)^{\tensor m}.\]

\begin{prop}\label{prop: twistor on M(k)}
We have the following:
\[\operatorname{kod} (M(k), J_{(a, b, 0)}) = \left\{ \begin{array}{ll}
0 & \text{if } (a, b, c) = (\pm 1, 0, 0),\\
-\infty & \text{otherwise}.
\end{array} \right.\]
\end{prop}

\begin{proof}
Let $f = u + \ii v$ be a smooth complex-valued function. The condition that $f (\varphi^1 \wedge \varphi^2)^{\tensor m}$ is a pseudo-holomorphic pluricanonical section translates in $\delbar f + mf \omega = 0$, which in turns is equivalent to
\begin{equation}\label{eq: system for M(k) twistor}
\left\{ \begin{array}{l}
\bar{\cX}_1(f) - \ii \frac{mk}{4}(b^2 - c^2)f = 0\\
\bar{\cX}_2(f) - \ii \frac{mk}{2} bc f = 0,
\end{array} \right.
\end{equation}
hence to
\[\left\{ \begin{array}{l}
b e_1(u) + c e_2(u) - a e_4(u) - e_3(v) + \frac{1}{2}mk(b^2 - c^2)v = 0\\
b e_1(v) + c e_2(v) - a e_4(v) + e_3(u) - \frac{1}{2}mk(b^2 - c^2)u = 0\\
c e_1(u) - b e_2(u) + a e_3(u) - e_4(v) + mkbcv = 0\\
c e_1(v) - b e_2(v) + a e_3(v) + e_4(u) - mkbcu = 0.
\end{array} \right.\]

For $c = 0$ and using the Fourier series expansion \eqref{eq: Fourier M(k)} we get the following system for the coefficients $u_I$ and $v_I$ (with $I = (A, B, C) \in \IZ^3$)
\begin{equation}\label{eq: system diff eq}
\left( \begin{array}{cccc}
2\pi \ii \pa{b e^{kz} \frac{1}{\delta} (A v_2 - Bu_2) - aC} & \frac{1}{2}mkb^2 & 0 & -1\\
-\frac{1}{2}mkb^2 & 2\pi \ii \pa{b e^{kz} \frac{1}{\delta} (A v_2 - Bu_2) - aC} & 1 & 0\\
-2\pi \ii b e^{-kz} \frac{1}{\delta} (-Av_1 + Bu_1) & -2\pi \ii C & a & 0\\
2\pi \ii C & -2\pi \ii b e^{-kz} \frac{1}{\delta} (-Av_1 + Bu_1) & 0 & a
\end{array} \right) \left( \begin{array}{c}
u_I\\
v_I\\
u'_I\\
v'_I
\end{array} \right) = 0
\end{equation}
where $'$ denotes the derivative with respect to $z$. The real part of the determinant of this matrix is
\[-4\pi^2 \pa{\frac{1}{\delta} e^{kz} a b (A v_2 - B u_2) + b^2 C}^2 - 4\pi^2 \frac{1}{\delta^2} e^{-2kz}b^2 (A v_1 - B u_1)^2 + \frac{1}{4} m^2 k^2 a^2 b^4\]
while its imaginary part is
\[2 \pi e^{-kz} \frac{1}{\delta} m k a b^3 (A v_1 - B u_1).\]
We analyze several different cases separately.
\begin{enumerate}
\item If $b = 0$ then $(a, b, c) = (\pm 1, 0, 0)$ and so $J_{(a, b, c)} = \pm J_0$. We computed that $\kod(M(k), J_{(\pm 1, 0, 0)}) = 0$ in Proposition \ref{prop: kod dim m(k) 1} (it is the fibre over the origin). From now on we will then assume that $b \neq 0$.
\item If $a \neq 0$ and $I = (A, B, C) \in \IZ^3$ is such that $Av_1 - Bu_1 \neq 0$, then the imaginary part of the determinant is non zero. In particular, the determinant itself is non zero and so $u_I = v_I = 0$ is the only solution.
\item If $a = 0$, then $b = \pm 1$ and the matrix of the system \eqref{eq: system diff eq} has determinant
\[-4\pi^2 C^2 - 4 \pi^2 \frac{1}{\delta^2}e^{-2kz}(A v_1 - B u_1)^2.\]
It follows that if $I = (A, B, C) \in \IZ^3$ is such that either $C \neq 0$ or $Av_1 - Bu_1 \neq 0$, then the solution of system \eqref{eq: system diff eq} is $u_I = v_I = 0$. We are then left with the case where $C = Av_1 - Bu_1 = 0$.\\
Assume first that $A v_1 - B u_1 = 0$ is possible only for $A = B = 0$. In this case we can conclude from our discussion that $u = u_{(0, 0, 0)}$, $v = v_{(0, 0, 0)}$ are constant. So the same is true for $f$, which must then be zero because of system \eqref{eq: system for M(k) twistor}. So $\kod(M(k), J_{(0, \pm 1, 0)}) = -\infty$ in this case.\\
Assume now that $C = A v_1 - B u_1 = 0$ for some $(A, B) \neq (0, 0)$. Then there exists a pair $(A', B') \neq (0, 0)$ such that $A' v_1 - B' u_1 = 0$, $\gcd(A', B') = 1$ and with the property that every other pair $(A, B) \in \IZ^2$ satisfying $Av_1 - Bu_1 = 0$ is of the form $(A, B) = h(A', B')$ for some $h \in \IZ$. Since $(A, B) \neq (0, 0)$ we have $A' v_2 - B' u_2 \neq 0$. The first two equations of \eqref{eq: system diff eq} are
\[\left\{ \begin{array}{l}
u'_I = \frac{1}{2} mk u_I \mp 2\pi \ii e^{kz} \frac{1}{\delta} h(A' v_2 - B' u_2) v_I\\
v'_I = \pm 2\pi \ii e^{kz} \frac{1}{\delta} h(A' v_2 - B' u_2) u_I + \frac{1}{2} mk v_I,
\end{array} \right.\]
which decouples after the substitution
\[\left( \begin{array}{c}
u_I\\
v_I
\end{array} \right) = \left( \begin{array}{cc}
\mp \ii & \pm \ii\\
1 & 1
\end{array} \right) \left( \begin{array}{c}
\xi_I\\
\zeta_I
\end{array} \right).\]
Explicitly, we find the system
\[\left\{ \begin{array}{l}
\xi'_I = \pa{\frac{1}{2} mk + 2\pi e^{kz} \frac{1}{\delta} h (A' v_2 - B' u_2)} \xi_I\\
\zeta'_I = \pa{\frac{1}{2} mk - 2\pi e^{kz} \frac{1}{\delta} h (A' v_2 - B' u_2)} \zeta_I
\end{array} \right.\]
whose solutions are
\[\left\{ \begin{array}{l}
\xi_I(z) = c_I e^{\frac{1}{2} mk z + 2\pi e^{kz} \frac{1}{k \delta} h (A' v_2 - B' u_2)}\\
\zeta_I(z) = d_I e^{\frac{1}{2} mk z - 2\pi e^{kz} \frac{1}{k \delta} h (A' v_2 - B' u_2)}
\end{array} \right. \qquad c_I, d_I \in \IC.\]
So we have
\[\left\{ \begin{array}{l}
u_I = \mp \ii \pa{\xi_I - \zeta_I}\\
v_I = \xi_I + \zeta_I
\end{array} \right. \Longrightarrow f_I = u_I + \ii v_I = \left\{ \begin{array}{ll}
2\ii \zeta_I & \text{for } b = 1,\\
2\ii \xi_I & \text{for } b = -1.
\end{array} \right.\]
We focus on the case $b = 1$, as the case $b = -1$ is similar. The Fourier expansion \eqref{eq: Fourier M(k)} simplifies to
\[f(x, y, z, t) = \sum_{h \in \IZ} 2 \ii d_h e^{\frac{1}{2} mk z - 2\pi e^{kz} \frac{1}{k \delta} h (A' v_2 - B' u_2)} e^{2 \pi \ii \frac{1}{\delta} h (A' v_2 - B' u_2) x}.\]
Now we exploit the fact that $f$ is equivariant with respect to the action of the lattice, in particular we have that $f(x, z) = f(e^{kn\gamma}x, z + n\gamma)$ for every $\gamma \in \IZ$ and we can write
\[f(e^{kn\gamma}x, z + n\gamma) = \sum_{h \in \IZ} g_h^\gamma(z) e^{2 \pi \ii \frac{1}{\delta} h (A' v_2 - B' u_2) x}.\]
The Fourier coefficients $g_h^\gamma$ can be computed:
\[\begin{array}{rl}
g_h^\gamma(z) = & \int_0^{\frac{\delta}{A' v_2 - B' u_2}} \sum_{l \in \IZ} 2 \ii d_l e^{\frac{1}{2} mk (z + n\gamma) - 2\pi e^{k(z + n\gamma)} \frac{1}{k \delta} l (A' v_2 - B' u_2)} e^{2 \pi \ii \frac{1}{\delta} (A' v_2 - B' u_2) (e^{kn\gamma l - h}) x} dx\\
= & \sum_{l \in \IZ} 2 \ii d_l e^{\frac{1}{2} mk (z + n\gamma) - 2\pi e^{k(z + n\gamma)} \frac{1}{k \delta} l (A' v_2 - B' u_2)} \int_0^{\frac{\delta}{A' v_2 - B' u_2}} e^{2 \pi \ii \frac{1}{\delta} (A' v_2 - B' u_2) (e^{kn\gamma l - h}) x} dx,
\end{array}\]
and because of the equivariance we have $f_h(z) = g_h^\gamma(z)$ for every $\gamma \in \IZ$. Let now $\gamma \rightarrow \pm \infty$ in such a way that $kn \gamma \rightarrow -\infty$: we have $f_h = \lim_{kn \gamma \rightarrow -\infty} g_h^\gamma$ and since
\[\begin{array}{l}
e^{\frac{1}{2} mk (z + n\gamma)} \longrightarrow 0\\
e^{- 2\pi e^{k(z + n\gamma)} \frac{1}{k \delta} l (A' v_2 - B' u_2)} \longrightarrow 1\\
\int_0^{\frac{\delta}{A' v_2 - B' u_2}} e^{2 \pi \ii \frac{1}{\delta} (A' v_2 - B' u_2) (e^{kn\gamma l - h}) x} dx \longrightarrow \left\{ \begin{array}{ll}
1 & \text{for } h = 0,\\
0 & \text{for } h \neq 0
\end{array} \right.
\end{array}\]
we deduce that this limit is zero. So $f = 0$ is the unique solution in this case.\\
As a consequence
\[\kod(M(k), J_{(0, \pm 1, 0)}) = -\infty\]
also in this case.
\item Finally, let $a \neq 0$ and assume that $A v_1 - B u_1 = 0$.\\
If $(A, B) \neq (0, 0)$ (observe that this is possible if and only if $u_1$ and $v_1$ are one a rational multiple of the other) then we have $Av_2 - Bu_2 \neq 0$ and so the determinant of the matrix of the system \eqref{eq: system diff eq}, which is
\[-4\pi^2 \pa{\frac{1}{\delta} e^{kz} ab h (A' v_2 - B' u_2) + b^2 C}^2 + \frac{1}{4} m^2 k^2 a^2 b^4,\]
is a function of $z$. It vanishes for at most a finite number of values of $z$. Where the determinant is non zero, the corresponding value of $u_I$ and $v_I$ must then be zero, and since both $u_I$ and $v_I$ are continuous we deduce that they must vanish identically.\\
We are then left with the case $(A, B) = (0, 0)$, where the determinant is
\[-4\pi^2 b^4 C^2 + \frac{1}{4} m^2 k^2 a^2 b^4.\]
For $C = 0$ this is non zero, hence $u_I = v_I = 0$. For $C \neq 0$ we deduce from the third equation of \eqref{eq: system diff eq} that
\[v_I = \frac{a}{2\pi \ii C} u'_I \Longrightarrow v'_I = \frac{a}{2\pi \ii C} u''_I,\]
so from the fourth equation we have
\[v'_I = -\frac{2\pi \ii C}{a} u_I \Longrightarrow -\frac{2\pi \ii C}{a} u_I = \frac{a}{2\pi \ii C} u''_I \Longrightarrow u''_I = \frac{4\pi^2 C^2}{a^2} u_I.\]
We can solve this last equation, and find that
\[\left\{ \begin{array}{l}
u_I(z) = c_I e^{\frac{2\pi C}{a} z} + d_I e^{-\frac{2\pi C}{a} z}\\
v_I(z) = -\ii \pa{c_I e^{\frac{2\pi C}{a} z} - d_I e^{-\frac{2\pi C}{a} z}}.
\end{array} \right.\]
for $c_I, d_I \in \IC$. It follows that the Fourier coefficient $f_I(z)$ for $f$ is
\[f_I(z) = u_I(z) + \ii v_I(z) = 2 c_I e^{\frac{2\pi C}{a} z}.
\]
It follows that
\[f = f(z, t) = \sum_{C \in \IZ \smallsetminus \{ 0 \}} 2 c_C e^{\frac{2\pi C}{a} z} e^{2\pi \ii C t}\]
and since $f(z, t) = f(z + n\gamma, t)$ for every $\gamma \in \IZ$ we deduce that
\[f_C(z) = f_C(z + n\gamma) \qquad \text{for every } \gamma \in \IZ.\]
Since this is possible only if $c_C = 0$, we deduce that we must have $f = 0$. It then follows that
\[\kod(M(k), J_{(a, b, 0)}) = -\infty.\]
\end{enumerate}
\end{proof}

\begin{rem}\label{rem: path in twistor}
We remark that the family of almost complex structures $\{ J_r \}$ is obtained as particular values of $a$, $b$, $c$, namely: $a = \alpha (1 - r)^2$, $b = -2 \alpha r^2$ and $ c = 2\alpha r (1 - r)$. Precisely $\{ J_r \}$ is a path in the Twistor space passing through $J_0$ ($r = 0$) and $-J_1$ ($r = 1$). However $J_2$ is not contained in this family.
\end{rem}

\subsubsection{Kodaira dimension of $(\cN, J_{(a, b, c)})$} Consider $(\cN, J_{(a, b, c)})$. 

We have that
\[\delbar \varphi^1 = \frac{\ii (1-a^2)}{4} (  \varphi^1 \wedge \bar{\varphi}^2  -  \varphi^2 \wedge \bar{\varphi}^1)\]
and
\[\delbar \varphi^2 = \frac{b}{2} \varphi^1 \wedge \bar{\varphi}^1 +\frac{1}{4} (c-ab\ii) \varphi^1 \wedge \bar{\varphi}^2 + \frac{1}{4} (c+ab\ii) \varphi^2 \wedge \bar{\varphi}^1),\]
from which
\[\delbar(\varphi^1 \wedge \varphi^2) = -\frac{1}{4} \pa{(c+ab\ii) \varphi^1 \wedge \varphi^2 \wedge \bar{\varphi}^1 + (1-a^2)\ii \varphi^1 \wedge \varphi^2 \wedge \bar{\varphi}^2}.\]
As a section of $\Omega^{0, 1}_X \tensor \Omega^{2, 0}_X$ this corresponds to $\delbar(\varphi^1 \wedge \varphi^2) = \omega \tensor (\varphi^1 \wedge \varphi^2)$, where
\[\omega = -\frac{1}{4} \pa{(c+ab\ii) \bar{\varphi}^1 + (1-a^2)\ii \bar{\varphi}^2}\]
and so
\[\delbar \pa{(\varphi^1 \wedge \varphi^2)^{\tensor m}} = m\omega \tensor (\varphi^1 \wedge \varphi^2)^{\tensor m}.\]
\begin{prop}
We have the following:
\[\operatorname{kod} (\cN, J_{(a, b, 0)}) = -\infty\]
for all $a,b.$
\end{prop}
\begin{proof} Following previous computation we have that a pluricanonical form  $f (\varphi^1 \wedge \varphi^2)^{\tensor m}$ is pseudo-holomorphic with respect to $J_{(a,b,c)}$ if and only if:
\[\left\{ \begin{array}{l}
\bar{\cX}_{1}(f)-\frac{1}{4}mf(c+ab\sqrt {-1})= 0\\
\bar{\cX}_{2}(f) - \frac{1}{4} \ii m(1-a^2) f = 0.
\end{array} \right.\]
Writing $f = u + \ii v$  previous system is equivalent to
\[\left\{ \begin{array}{l}
be_1(u) +ce_2(u)-ae_4(u)-e_3(v)- \frac{1}{2}m(cu -abv)=0\\
be_1(v)+ce_2(v) + e_3(v) - a e_4(v) - \frac{1}{2} m(abu + cv)=0\\
ae_3(u)+ ce_1(u) - be_2(u) -e_4(v) + \frac{1}{2} m(1-a^2)v=0\\
e_4(u) +c e_1(v) - b e_2(v)+ a e_3(v) -\frac{1}{2} m(1-a^2)u=0.
\end{array} \right.\]
By substituting Fourier coefficients $u_I(x)$ and $v_I(x)$, where $I=(A,B,C)\in {\mathbb Z}^3$, denoting $'$ the derivative with respect to $x$ we get:
\[\left\{ \begin{array}{l}
bu_I' +[\pi c \sqrt {-1}(A+xB+\frac{1}{2}x^2C)-aC-\frac{1}{2}mc]u_I+[\frac{1}{2}mab-\pi \sqrt{-1}(xC+B)]v_I=0\\
bv_I' +[\pi c \sqrt {-1}(A+xB+\frac{1}{2}x^2C)-aC-\frac{1}{2}mc]v_I-[\frac{1}{2}mab-\pi \sqrt{-1}(xC+B)]u_I=0\\
cu_I'-\pi  \sqrt {-1}[b(A+xB+\frac{1}{2}x^2C)-a(B+Cx)]u_I+[\frac{1}{2}m(1-a^2)-\pi \sqrt{-1}C]v_I=0\\
cv_I'-\pi  \sqrt {-1}[b(A+xB+\frac{1}{2}x^2C)-a(B+Cx)]v_I-[\frac{1}{2}m(1-a^2)-\pi \sqrt{-1}C]u_I=0.
\end{array} \right.\]
If $c=0$ the last two equations involve only $u_I$ and $v_I$, the determinant of the matrix of this homogeneous system has imaginary part given by:
$$\pi Cm(1-a^2)$$
which is zero if and only if $C=0$. Moreover for $C=0$ the real part of this determinant vanishes only for finite values of $x$ and we can conclude that $u_I=v_I=0$. So the statement.

\end{proof}

\subsubsection{Kodaira dimension of $(\cM(\lambda), J_{(a, b, c)})$}  Finally consider $(\cM(\lambda), J_{(a, b, c)})$. 

We have that
\[\delbar \varphi^1 = \frac{\lambda b}{2}  \varphi^1 \wedge \bar{\varphi}^1 + \frac{1}{4}\pa{ (1+\lambda )c + ab(1-\lambda)\ii)  \varphi^1 \wedge \bar{\varphi}^2+ (\lambda -1)(c +ab\ii) } \varphi^2 \wedge \bar{\varphi}^1)\]
and
\[\delbar \varphi^2 = \frac{(2+\lambda)}{4} (-b+ac\ii) \varphi^1 \wedge \bar{\varphi}^2 -\frac{1}{4}(b\lambda +ac (2+\lambda) \ii)\varphi^2 \wedge \bar{\varphi}^1 - \frac{c(1+\lambda)}{2}\varphi^2 \wedge \bar{\varphi}^2,\]
from which
\[\delbar(\varphi^1 \wedge \varphi^2) = \frac{1}{4} \pa{(-b\lambda+ac(2+\lambda)\ii) \varphi^1 \wedge \varphi^2 \wedge \bar{\varphi}^1 + (c(1+\lambda)-ab(1-\lambda)\ii) \varphi^1 \wedge \varphi^2 \wedge \bar{\varphi}^2}.\]
As a section of $\Omega^{0, 1}_X \tensor \Omega^{2, 0}_X$ this corresponds to $\delbar(\varphi^1 \wedge \varphi^2) = \omega \tensor (\varphi^1 \wedge \varphi^2)$, where
\[\omega = \frac{1}{4} \pa{(-b \lambda +ac(2+\lambda)\ii)\bar{\varphi}^1 + (c(1+\lambda)-ab(1-\lambda)\ii) \bar{\varphi}^2}.\]
and so
\[\delbar \pa{(\varphi^1 \wedge \varphi^2)^{\tensor m}} = m\omega \tensor (\varphi^1 \wedge \varphi^2)^{\tensor m}.\]

\begin{prop}
We have the following:
\[\operatorname{kod} (\cM(\lambda), J_{(a, b,0)}) = -\infty\]
for all $a,b$.
\end{prop}

\begin{proof}
Following previous computation we have that a pluricanonical form  $f (\varphi^1 \wedge \varphi^2)^{\tensor m}$ is pseudo-holomorphic with respect to $J_{(a,b,c)}$ if and only if:
\[\left\{ \begin{array}{l}
\bar{\cX}_{1}(f) + \frac{1}{4}mf(-b\lambda+ac\sqrt {-1}(2+\lambda))= 0\\
\bar{\cX}_{2}(f) + \frac{1}{4}mf(c(1+\lambda)-ab(1-\lambda) \ii ) = 0.
\end{array} \right.\]
Writing $f = u + \ii v$ and substituting Fourier coefficients $u_I(t)$ and $v_I(t)$, where $I=(A,B,C)\in {\mathbb Z}^3$ (cf.~\eqref{eq: Fourier M(lambda)}), and denoting $'$ the derivative with respect to $t$ we get:
\[\left\{ \begin{array}{l}
2 b u_I' + [4\pi \sqrt{-1}(c e^t \mu_I - a \nu_I e^{-(1+\lambda)t}) - m b \lambda] u_I - [4\pi \sqrt{-1} \lambda_I e^{\lambda t} + m a c (2 + \lambda)] v_I = 0\\
2 b v_I' + [4\pi \sqrt{-1} (c e^t \mu_I - a \nu_I e^{-(1+\lambda)t}) - m b \lambda] v_I + [4\pi \sqrt{-1} \lambda_I e^{\lambda t} + m a c (2 + \lambda)] u_I = 0\\
2 c u_I' - [4\pi \sqrt{-1} (b \mu_I e^t - a \lambda_I e^{\lambda t} - m c (1 + \lambda)] u_I + [-4\pi \sqrt{-1} \nu_I e^{-(1+\lambda)t} + a b (1 - \lambda)m] v_I = 0\\
2 c v_I' - [4\pi \sqrt{-1} (b \mu_I e^t - a \lambda_I e^{\lambda t} - m c (1 + \lambda)] v_I - [-4\pi \sqrt{-1} \nu_I e^{-(1+\lambda)t} + a b (1 - \lambda)m] u_I = 0.
\end{array} \right.\]
If $c=0$ the last two equations involve only $u_I$ and $v_I$, the determinant of the matrix of this homogeneous system has imaginary part given by:
\[-8\pi a b \nu_I m (1-\lambda) e^{-(1+\lambda)t}\]
which is zero if and only if $a \nu_I = 0$. For $\nu_I = 0$ the real part of this determinant vanishes only for finite values of $t$ and so we can conclude that $u_I = v_I = 0$. Finally, assume that $\nu_I \neq 0$ and $a = 0$. In this case we have $b = \pm 1$ and we can then deduce from the last two equations that
\[u_I = \pm \frac{\mu_I}{\nu_I} e^{(2 + \lambda)t} v_I, \qquad v_I = \mp \frac{\mu_I}{\nu_I} e^{(2 + \lambda)t} u_I.\]
It follows that $u_I = \frac{\mu_I^2}{\nu_I^2} e^{2(2 + \lambda)t} u_I$ and so $u_I = 0$, and similarly we deduce that $v_I = 0$.

Hence we get the statement.
\end{proof}

\section{Norden structures}\label{sect: norden}

Norden structures were introduced by Norden in \cite{Norden} and then studied also as almost complex structures with $B$-metric and anti K\"ahlerian structures, they have applications in mathematics and in theoretical physics. We recall here the definition.

\begin{defin}
Let $(M, J)$ be an almost complex manifold and let $g$ be a pseudo Riemannian metric on $M$ such that $J$ is a $g$-symmetric operator, $(J, g)$ is called \emph{Norden structure} on $M$ and $(M, J, g)$ is called a \emph{Norden manifold}. If $J$ is integrable then $(M, J, g)$ is called a \emph{complex Norden manifold}.
\end{defin}

Let $M$ one of the three four dimensional solvmanifolds $M(k)$, $\cN$, $\cM(\lambda)$.

By using previous notations consider the following natural neutral  pseudo Riemannian metrics on $M$:
\[\begin{array}{l}
\tilde{g}_0 := e^1 \otimes e^1 + e^2 \otimes e^2 - e^3 \otimes e^3 - e^4 \otimes e^4,\\
\tilde{g}_1 := e^1 \otimes e^1 - e^2 \otimes e^2 + e^3 \otimes e^3 - e^4 \otimes e^4,\\
\tilde{g}_2 := e^1 \otimes e^1 - e^2 \otimes e^2 - e^3 \otimes e^3 + e^4 \otimes e^4.
\end{array}\]

Direct computation gives the following.

\begin{lemma}
$\tilde{g}_0(J_0 \blank, J_0\blank) = \tilde{g}_0(\blank, \blank)$  and
$\tilde{g}_0(J_i \blank, \blank) = \tilde{g}_0(\blank, J_i \blank)$ for $i = 1, 2$.\\

$\tilde{g}_1(J_1 \blank, J_1\blank) = \tilde{g}_1(\blank, \blank)$  and
$\tilde{g}_1(J_i \blank, \blank) = \tilde{g}_1(\blank, J_i \blank)$ for $i = 0, 2$.\\

$\tilde{g}_2(J_2 \blank, J_2\blank) = \tilde{g}_2(\blank, \blank)$  and
$\tilde{g}_2(J_i \blank, \blank) = \tilde{g}_2(\blank, J_i \blank)$ for $i = 0, 1$.
\end{lemma}

Hence we get:
\begin{cor} $(M, J_i, \tilde{g}_0)$ is a Norden manifold for $i = 1, 2$ and $(M, J_0, \tilde{g}_0)$ is a pseudo Hermitian manifold.\\
$(M, J_i, \tilde{g}_1)$ is a Norden manifold for $i = 0, 2$ and $(M, J_1, \tilde{g}_1)$ is a pseudo Hermitian manifold.\\
$(M, J_i, \tilde{g}_2)$ is a Norden manifold for $i = 0, 1$ and $(M, J_2, \tilde{g}_2)$ is a pseudo Hermitian manifold.
\end{cor}

Moreover we define: 
\[
\begin{array}{l}
\tilde{\omega}_0(\blank,\blank) := \tilde{g}_0(J_0 \blank,\blank) = e^1 \wedge e^2 - e^3 \wedge e^4,\\
\tilde{\omega}_1(\blank,\blank) := \tilde{g}_1(J_1 \blank,\blank) = e^1 \wedge e^3 + e^2 \wedge e^4,\\
\tilde{\omega}_2(\blank,\blank) := \tilde{g}_2(J_2 \blank,\blank) = e^1 \wedge e^4 - e^2 \wedge e^3,
\end{array}\]
and we easily get:

\begin{lemma}  If $M$=\texorpdfstring{$M(k)$}{M(k)} then $d\tilde{\omega}_0 = 0$, $d\tilde{\omega}_1 = -k e^2 \wedge e^3 \wedge e^4$, $d\tilde{\omega}_2 = k e^1 \wedge e^3 \wedge e^4$.
If $M$= \texorpdfstring{$\cN$}{M} then $d\tilde\omega_0 =e^1\wedge e^3\wedge e^4, \, d\tilde\omega_1 =-e^1 \wedge e^2 \wedge e^3, \,d\tilde\omega_2 =0$.\\
If $M$=\texorpdfstring{$\cM(\lambda)$}{M(l)} then $d\tilde\omega_0 =-e^1\wedge e^3\wedge e^4, \, d\tilde\omega_1 =\lambda e^1 \wedge e^2 \wedge e^4, \,d\tilde\omega_2 =(1+\lambda)e^1 \wedge e^2 \wedge e^3$.
\end{lemma}

Finally, by direct computation, we obtain the following expressions for twin metrics:
\[\begin{array}{lclcl}
\hat{g}_{01}(\blank,\blank) & := & \tilde{g}_0(\blank, J_1 \blank) & = & -e^1 \otimes e^3 - e^3 \otimes e^1 + e^2 \otimes e^4 +e^4 \otimes e^2,\\
\hat{g}_{02}(\blank,\blank) & := & \tilde{g}_0(\blank, J_2 \blank) & = & -e^1 \otimes e^4 - e^4 \otimes e^1 - e^2 \otimes e^3 - e^3 \otimes e^2,\\
\hat{g}_{10}(\blank,\blank) & := & \tilde{g}_1(\blank, J_0 \blank) & = & -e^1 \otimes e^2 - e^2 \otimes e^1 - e^3 \otimes e^4 -e^4 \otimes e^3,\\
\hat{g}_{12}(\blank,\blank) & := & \tilde{g}_1(\blank, J_2 \blank) & = & -e^1 \otimes e^4 - e^4 \otimes e^1 + e^2 \otimes e^3 +e^3 \otimes e^2,\\
\hat{g}_{20}(\blank,\blank) & := & \tilde{g}_2(\blank, J_0 \blank) & = & -e^1 \otimes e^2 - e^2 \otimes e^1 + e^3 \otimes e^4 + e^4 \otimes e^3,\\
\hat{g}_{21}(\blank,\blank) & := & \tilde{g}_2(\blank, J_1 \blank) & = & -e^1 \otimes e^3 - e^3 \otimes e^1 - e^2 \otimes e^4 - e^4 \otimes e^2.
\end{array}\]

\bibliographystyle{alpha}
\bibliography{KodDimTwoDimSolv}

\end{document}